\documentclass[11pt]{article}  

\usepackage{amssymb}
\usepackage{bm}
\usepackage{amsfonts}
\usepackage{latexsym}
\usepackage{amsthm}
\usepackage{hyperref}
\usepackage{xcolor} 
\hypersetup{
	colorlinks,
    linkcolor=black,
    citecolor={blue!50!black},
    urlcolor={blue!80!black},
    urlcolor={black}
}
\usepackage{enumerate}
\usepackage{epsfig}
\usepackage{graphicx}
\usepackage{color}
\usepackage{float} 
\usepackage{subfig}
\usepackage{amsmath,bm}
\usepackage{amssymb}
\usepackage{stmaryrd}
\usepackage{mathrsfs}
\usepackage{makeidx}
\usepackage{fancyhdr}
\usepackage{lastpage}
\usepackage{url}
\usepackage{geometry}
\usepackage{bbm}
\usepackage{fullpage}
\usepackage{amsmath}
\usepackage{hyperref}
\usepackage{parskip}
\usepackage{booktabs}
\usepackage{multirow}
\usepackage{comment}
\usepackage{mathtools}
\usepackage{slashbox}
\usepackage{diagbox}

\usepackage{csquotes}

\usepackage{tikz}

\usepackage[titletoc,toc]{appendix}

\usepackage{subfloat}

\restylefloat{table}

\newcommand{\argmin}{\operatornamewithlimits{argmin}}

\def\qed{\hfill $\Box$}
\renewenvironment{proof}{\vspace{.01cm}   \noindent{\bf Proof }}{\qed \vspace{.1cm}}

\newtheoremstyle{mythmstyle}
{1em} 
{0.1em} 
{\itshape} 
{} 
{\bfseries} 
{} 
{1em} 
{} 

\theoremstyle{mythmstyle}
\newtheorem{theorem}{Theorem}[section]
\newtheorem{corollary}{Corollary}[section]

\newtheorem{lemma}{Lemma}[section]
\newtheorem{proposition}{Proposition}[section]

\newtheorem{remark}{Remark}[section]
\newtheorem{conjecture}{Conjecture}[section]

\makeatletter \@addtoreset{equation}{section}

\usepackage[noabbrev]{cleveref}

\begin{document}
\pagenumbering{gobble} 
\title{\LARGE {\bf Analysis of sparse grid multilevel estimators for multi-dimensional Zakai equations}}
\author{Christoph Reisinger\footnote{Mathematical Institute, University of Oxford, Andrew Wiles Building, Woodstock Road, Oxford, OX2 6GG, UK, E-mail: christoph.reisinger@maths.ox.ac.uk}
\ and Zhenru Wang\footnote{Mathematical Institute, University of Oxford, Andrew Wiles Building, Woodstock Road, Oxford, OX2 6GG, UK, E-mail: zhenru.wang@maths.ox.ac.uk}    
 }

\date{}
\maketitle
\setcounter{secnumdepth}{3}
\setcounter{tocdepth}{2}
\pagenumbering{arabic}

\abstract{In this article, we analyse the accuracy and computational complexity of estimators for expected functionals
of the solution to multi-dimensional parabolic stochastic partial differential equations (SPDE) of Zakai-type. Here, we use the
Milstein scheme for time integration and an alternating direction implicit (ADI) splitting of the spatial finite difference discretisation, coupled with the
sparse grid combination technique and multilevel Monte Carlo sampling (MLMC).
In the two-dimensional case, we find
by detailed 
Fourier analysis that for a root-mean-square error (RMSE)~$\varepsilon$, MLMC on sparse grids has the optimal complexity $O(\varepsilon^{-2})$, whereas MLMC on regular grids has $O(\varepsilon^{-2}(\log\varepsilon)^2)$, standard MC on sparse grids $O(\varepsilon^{-7/2}(|\log\varepsilon|)^{5/2})$, and MC on regular grids $O(\varepsilon^{-4})$. Numerical tests confirm these findings empirically.
We give a discussion of the higher-dimensional setting without detailed proofs, which suggests that MLMC on sparse grids always leads to the optimal complexity,
standard MC on sparse grids has a fixed complexity order independent of the dimension (up to a logarithmic term), whereas the cost of MLMC and MC on regular grids increases exponentially with the dimension.
}

\section{Introduction}

%

The focus of this paper is the efficient simulation of the two-dimensional SPDE
\begin{eqnarray}\label{eq_SPDE}
\mathrm{d}v &=& Lv \,\mathrm{d}t - \sqrt{\rho_x}\frac{\partial v}{\partial x}\,\mathrm{d}W_t^x - \sqrt{\rho_y}\frac{\partial v}{\partial y}\,\mathrm{d}W_t^y
\end{eqnarray}
for $x,y\in\mathbb{R},\ 0<t\leq T$,
subject to the Dirac initial datum
\begin{equation}\label{eq_2dDiracInitial}
v(0,x,y) = \delta(x-x_0) \otimes \delta(y-y_0)
\end{equation}
for given $x_0$ and $y_0$,
where 
$W=(W^x,\,W^y)$ is a two-dimensional standard Brownian motion with correlation $\rho_{xy}$ on a probability space
$(\Omega,\mathcal{F},\mathbb{P})$,
\begin{eqnarray}
Lv &=& -\mu_x\frac{\partial v}{\partial x} -\mu_y\frac{\partial v}{\partial y} + \frac{1}{2}\bigg(\frac{\partial^2 v}{\partial x^2} + 2\sqrt{\rho_x\rho_y}\rho_{xy}\frac{\partial^2v}{\partial x\partial y} + \frac{\partial^2 v}{\partial y^2} \bigg),
\nonumber
\end{eqnarray}
$\mu_x,\mu_y$ and $0\leq\rho_x,\rho_y< 1$, $-1\leq\rho_{xy}\leq 1$ are real-valued parameters.

This is a special case of the Zakai equation from stochastic filtering where $v$ describes the distribution of the filter given a signal process $W$ (see \cite{bain2009fundamentals,gobet2006discretization}).

A classical result states that, for a class of SPDEs including \eqref{eq_SPDE}, with initial condition in $L_2$, there exists a unique solution $v\in L_2(\Omega\times(0,T), \mathcal{F}, L_2(\mathbb{R}^2))$ \cite{krylov1981stochastic}. This does not include Dirac initial data \eqref{eq_2dDiracInitial}, but in fact, the solution to \eqref{eq_SPDE} and \eqref{eq_2dDiracInitial} can be analytically derived as the smooth (in $x$ and $y$) function
\begin{equation}\label{eq_2dTheoreticalResult}
v(T,x,y) = \frac{\exp\Big(-\frac{\big(x-x_0-\mu_x T-\sqrt{\rho_x}W_T^x\big)^2}{2(1-\rho_x)T}-\frac{\big(y-y_0-\mu_y T-\sqrt{\rho_y}W_T^y\big)^2}{2(1-\rho_y)T}\Big)}{2\pi\sqrt{(1-\rho_x)(1-\rho_y)}\,T}\,.
\end{equation}

More commonly, however, such a closed-form solution is not available, for instance in the case of variable coefficients.
We will focus on \eqref{eq_SPDE} for the analysis, but the numerical methods we investigate apply similarly to such a wider class. 

There is a large body of recent literature on the numerical solution of SPDEs. Most closely related to the present work,
Giles and Reisinger in \cite{ref2} used an explicit Milstein finite difference approximation to the solution of the one-dimensional SPDE 
\begin{equation}\label{eq_1dspde}
\mathrm{d}v = -\mu\frac{\partial v}{\partial x}\,\mathrm{d}t + \frac{1}{2}\frac{\partial^2 v}{\partial x^2}\,\mathrm{d}t - \sqrt{\rho}\frac{\partial v}{\partial x}\,\mathrm{d}W_t,\qquad (t,x)\in(0,T)\times\mathbb{R},
\end{equation}
where $T>0$, $W$ is a standard Brownian motion, and $\mu$ and $0\le \rho<1$ are real-valued parameters; \cite{ref1} extended the discretisation to an implicit method on the basis of the $\sigma$-$\theta$ time-stepping scheme.
This 1-d SPDE \eqref{eq_1dspde} has also been used to model default risk in large credit portfolios (see \cite{bush2011stochastic}).

For the 2-d SPDE \eqref{eq_SPDE}, we will use an implicit method such that under some constraints on $\rho_x,\rho_y,\rho_{xy}$, the scheme is unconditionally mean-square stable. Furthermore, we use an Alternating Direction Implicit (ADI) factorisation which is more convenient computationally than a purely implicit scheme, and is also unconditionally mean-square stable (see \cite{reisinger2018stability}).

We consider the following functional of the solution,
\begin{equation}
\label{eq_2dlossfunctional}
P_t = \int_0^\infty\int_0^\infty v(t,x,y)\,\mathrm{d}x\,\mathrm{d}y, 
\end{equation}
which is a two-dimensional version of the loss in \cite{bush2011stochastic, ref2} which represents the proportion of the defaulted firms there.
In the context of filtering, $P_t$ is related to the cumulative distribution function of the filter given the signal.

The functional in \eqref{eq_2dlossfunctional} is a special case of more general linear and nonlinear functionals of the form
\begin{equation*}
\int_{-\infty}^\infty\int_{-\infty}^\infty f(x,y) v(t,x,y)\,\mathrm{d}x\,\mathrm{d}y,
\qquad g\Big( \int_{0}^\infty\int_{0}^\infty v(t,x,y)\,\mathrm{d}x\,\mathrm{d}y\Big),
\end{equation*}
with $f$ being the Heaviside function and $g$ the identity in the case of $P_t$.
Preliminary derivations indicate that our analysis may be extended from $P_t$ to these cases for sufficiently smooth $f$ and $g$, by judicious multivariate Taylor expansion,
but this involves exceedingly lengthy calculus and is beyond the scope of this work.



A classical approach to approximating $\mathbb{E}[P_t]$ is the standard Monte Carlo method,
using a suitable approximation scheme for \eqref{eq_SPDE} and sampling of $W$ on a discrete time mesh. 
For the SPDE \eqref{eq_SPDE} and standard schemes, to achieve a root mean square error (RMSE) $\varepsilon$, this requires an overall computational cost $O(\varepsilon^{-4})$,
as we require $O(\varepsilon^{-2})$ samples, $O(\varepsilon^{-1})$ time steps (e.g., for the Euler-Maruyama scheme with weak order 1), and $O(\varepsilon^{-1/2})$ mesh points in each direction (e.g., for central difference schemes with order 2). One way to reduce the cost is the MLMC method (see \cite{giles2008multilevel}) by using the SPDE solution on paths with a coarse timestep and spatial mesh as a control variate of solutions on paths with a fine timestep and mesh.
As a result, for a fixed accuracy $\varepsilon$, the cost can be reduced significantly to $O(\varepsilon^{-2}(\log \varepsilon)^2)$ by standard MLMC methodology as in \cite{ref2}. However, this complexity of the MLMC method is not optimal as in the one-dimensional case in \cite{ref2}. The reason is that the cost of each sample on higher levels increases with the same order as the variance decays. 
Moreover, the total cost of MLMC increases exponentially in the dimension.

The approach taken here is to approximate the SPDE \eqref{eq_SPDE} by the sparse grid combination technique. Sparse grids were first introduced to solve high-dimensional PDEs on a tensor product space in \cite{zenger1997sparse}. 
The error bounds in \cite{bungartz1998finite} show that sparse finite element approximations can alleviate the curse of dimensionality in the numerical implementation of certain elliptic PDEs with sufficiently smooth solutions. In contrast to the finite element method, the combination technique, first introduced in \cite{griebel1990combination}, decomposes the solution into contributions from simple tensor product grids with different resolutions in each dimension. 
For a survey of methods and early results see \cite{bungartz2004sparse}.
The analysis in \cite{pflaum1999error} and \cite{reisinger2012analysis} shows that the computational cost for the combination technique applied to the Poisson problem with sufficient regularity is independent of the dimension, up to a logarithmic term, for finite elements and finite differences, respectively.
Hendricks, Ehrhardt and G{\"u}nther combine the sparse grid combination technique with the ADI scheme for diffusion equations in \cite{hendricks2016high}.

This sparse grid method has been extended to multi-index Monte Carlo (MIMC) in the context of SPDEs in \cite{ref5}. MIMC can be viewed as the sparse grid combination technique applied to equations with stochasticity, with optimised number of samples for the individual terms in the combination formula (the ``hierarchical surpluses''), akin MLMC. Giles, Kuo and Sloan summarised these ideas applied to elliptic PDEs with finite-dimensional uncertainty in the coefficients in \cite{giles2017combining}.
The MIMC method was applied in \cite{reisinger2018analysis} to a 1-d SPDE \eqref{eq_1dspde}, where the timestep and space mesh are coupled for stability. However, optimal complexity is not achieved in this space-time method. 

In this paper, given a fixed  timestep and Brownian path, we solve the SPDE using the sparse grid combination technique in space.
Then, to evaluate $\mathbb{E}[P_t]$, we use $M$ independent samples of the hierarchical surpluses 
 and calculate the average. The benefit here is that with a RMSE $\varepsilon$, the total cost is fixed with $O(\varepsilon^{-7/2})$ up to a logarithmic term, as the cost for one sample is $O(\varepsilon^{-3/2}|\log\varepsilon|^{5(d-1)/2})$, and the number of samples needed is $M = O(\varepsilon^{-2})$. Hence, this will improve on the complexity of the MLMC method, whose total cost is $O(\varepsilon^{-1-d/2})$, when the dimension $d>5$.

To recover the optimal complexity $O(\varepsilon^{-2})$, we further combine the sparse combination technique and MLMC, in a different way from standard MIMC.
In this way, the total cost is $O(\varepsilon^{-2})$ independent of dimension.

The rest of this article is structured as follows. We define the approximation schemes in Section~\ref{sec_approxandmainresults}. Section \ref{sec_2dMIMCEstimator} gives a Fourier analysis of the sparse combination estimators. 
Section \ref{sec_2dNumerical} shows numerical experiments confirming the above findings, and
Section \ref{sec_highdim} generalises the problem to higher dimensions. Section~\ref{sec_Conclusion} offers conclusions and directions for further research.

\section{Approximation and main results}\label{sec_approxandmainresults}

For simplicity of presentation, we initially restrict ourselves with the description of the schemes and their analysis 
to the two-dimensional case. The extension to higher dimensions is discussed in Section~\ref{sec_highdim}.

Moreover, we focus on the case of constant coefficients as in~\eqref{eq_SPDE} and the functional~\eqref{eq_2dlossfunctional}.
While the numerical method itself is directly applicable to the variable coefficient case 
and more general functionals, the Fourier analysis we perform is tailored to the present setting.

\subsection{Semi-implicit Milstein finite difference scheme}\label{sec_2dImplicitMilsteinScheme}

We use a spatial grid with uniform spacing $h_x,\,h_y>0$, and let $V_{i,j}^{n}$ be the approximation to $v(nk,ih_x,jh_y)$, $n=1,\ldots,N$, $i,j\in\mathbb{Z}$. 
We assume for simplicity that $i_0:=x_0/h_x$ and $j_0:=y_0/h_y$ are integers.
Then we approximate $v(0,x,y) = \delta(x-x_0) \otimes \delta(y-y_0)$
by
\begin{equation}\label{eq_2dDiracInitialApprox}
V_{i,j}^0 = h_x^{-1}h_y^{-1}\delta_{(i_0,\,j_0)} = 
\begin{cases}
h_x^{-1}h_y^{-1},\quad &i=i_0,\ j=j_0,\\
0,&\text{otherwise}.
\end{cases}
\end{equation}

We use the semi-implicit Milstein scheme to approximate \eqref{eq_SPDE}, proposed in \cite{reisinger2018stability}:
\begin{equation}\label{eq_2DimplicitMilstein}
\begin{aligned}
\!\!\!\!\bigg(I + \frac{\mu_x k}{2h_x}D_x + \frac{\mu_y k}{2h_y}D_y - \frac{k}{2h_x^2}D_{xx} - \frac{k}{2h_y^2}D_{yy}\!\bigg)V^{n+1}\!
= \bigg(\! I \!- \!\frac{\sqrt{\rho_x k}Z_{n,x}}{2h_x}D_x \!-\! \frac{\sqrt{\rho_y k}\widetilde{Z}_{n,y}}{2h_y}D_y & \\
 + \frac{\rho_xk(Z_{n,x}^2\!\!-\!1)}{8h_x^2}D_x^2 \!+\! \frac{\rho_yk(\widetilde{Z}_{n,y}^2\!\!-\!1)}{8h_y^2}D_y^2 \!+\!
\frac{\sqrt{\rho_x\rho_y}kZ_{n,x}\widetilde{Z}_{n,y}}{4h_xh_y}D_{xy}\!\bigg)\!V^n\!& \\
=: M_n V^n, \hspace{2.3 cm}  &
\end{aligned}
\vspace{-0.2 cm}
\end{equation}
where $M$ is a random operator and
\begin{align*}
&(D_xV)_{i,j} = V_{i+1,j}-V_{i-1,j},\qquad\qquad\qquad\quad\ \;
(D_yV)_{i,j} = V_{i,j+1}-V_{i,j-1},\\
&(D_{xx}V)_{i,j} = V_{i+1,j}-2V_{i,j}+V_{i-1,j},\quad\qquad\quad\;\,
(D_{yy}V)_{i,j} = V_{i,j+1}-2V_{i,j}+V_{i,j-1},\\
&(D_{xy}V)_{i,j} = V_{i+1,j+1}-V_{i-1,j+1}-V_{i+1,j-1}+V_{i-1,j-1},
\end{align*}
and $\widetilde{Z}_n^y = \rho_{xy}Z_n^x + \sqrt{1-\rho_{xy}^2}Z_n^y$ with $Z_n^x,Z_n^y\sim N(0,1)$ independent normals. 
Briefly, the terms on the left-hand side of \eqref{eq_2DimplicitMilstein} correspond to the implicit approximation of the operator $L$ in \eqref{eq_SPDE};
the second and third terms on the right-hand side are the Euler-Maruyama approximation of the stochastic integral; and the last three terms the Milstein correction for strong first order 1. Note that the cross-derivative term on the left-hand side cancels out with a Milstein term, as detailed in \cite{reisinger2018stability}.

To save computational cost, we combine the scheme with an Alternating Direction Implicit (ADI) factorisation \cite{ref6}, 
\begin{equation}\label{eq_ADIdifference}
\begin{aligned}
\bigg(I + \frac{\mu_x k}{2h_x}D_x - \frac{k}{2h_x^2}D_{xx}\bigg)\bigg(I + \frac{\mu_y k}{2h_y}D_y - \frac{k}{2h_y^2}D_{yy}\bigg)V^{n+1}
= M_n V^n\!.
\end{aligned}
\end{equation}


A detailed analysis of the $L_2$ stability and convergence of these schemes can be found in \cite{reisinger2018stability}.
Here, we state the main result which is relevant here.
We make the following technical assumptions throughout\footnote{They are used in the proof of Lemma \ref{lem_2dSGridhighwave}.}:
\begin{align}
\label{eq_stablerhos}
2\rho_x^2(1+2|\rho_{xy}|) &\leq 1, \quad 
2\rho_y^2(1+2|\rho_{xy}|) &\leq 1, \quad 
2\rho_x\rho_y(3\rho_{xy}^2 + 2|\rho_{xy}|+1 ) &\leq 1. 
\end{align}

\begin{theorem}[Corollary 2.1 in \cite{reisinger2018stability}]\label{rmk_2dImplicitConvergence}
Let $\lambda>0$ be fixed and \eqref{eq_stablerhos} be satisfied. Then there exists
$C>0$ such that for any $k,h_{x}, h_y$ with
\begin{equation}\label{eq_kcondition}
k/\min\{h_x^2, h_y^2\} \leq \lambda,
\end{equation}
for the solution to the implicit Milstein scheme \eqref{eq_2DimplicitMilstein}, 
\[
\sqrt{\mathbb{E}\big[| V_{i,j}^N-v(T,ih_x,jh_y)|^2\big]} \le C (h_x^2 + h_y^2).
\]
\end{theorem}

This convergence result also holds for the ADI scheme \eqref{eq_ADIdifference}. 


\subsection{Sparse combination estimators}\label{sec_SparseGridsApprox}

Now we consider the specific functional
\begin{equation}\label{eq_2dLossTheoretical}
P_T = \int_0^\infty\int_0^\infty v(T,x,y)\,\mathrm{d}x\,\mathrm{d}y,
\end{equation}
as discussed in the introduction, where $v$ is the solution to \eqref{eq_SPDE} and \eqref{eq_2dDiracInitial}.

By introducing integer multi-indices $\bm{l} = (l_1,l_2)\in\mathbb{N}^2$ as the refinement levels of the spatial mesh in dimensions $x$ and $y$, respectively, we denote by
$P_{(l_1,\,l_2)}^N$ the discrete approximations to $P$ with mesh sizes $h_x = h_{0}\cdot2^{-l_1}$, $h_y = h_{0}\cdot2^{-l_2}$, and fixed timestep~$k$
with $k \leq \lambda \min\{h_x^2,h_y^2\}$. Let $N = T/k$. 
We then use the trapezoidal approximation
\begin{equation}\label{eq_2dloss}
\begin{aligned}
P_{(l_1,\,l_2)}^N &= h_xh_y\sum_{i=1}^{\infty}\sum_{j=1}^{\infty} V_{i,j}^N + \frac{h_x}{2}\sum_{j=1}^{\infty}V_{0,j}^N + \frac{h_y}{2}\sum_{i=1}^{\infty}V_{i,0}^N + \frac{1}{4}h_xh_y V_{0,0}^N, 
\end{aligned}
\end{equation}
where $V_{i,j}^N$ is the solution to \eqref{eq_2DimplicitMilstein}.


\begin{proposition}\label{prop_2dErrorofLoss}
Assume \eqref{eq_stablerhos} holds.
Let $P$ be given by \eqref{eq_2dLossTheoretical} 
and $P^N_{(l_1,\,l_2)}$ by \eqref{eq_2dloss}. 
Then for fixed $\lambda>0$ there exists $C>0$ such that
for any $(l_1,l_2)\in \mathbb{N}_0^2$, $k \leq \lambda \min\{h_x^2,h_y^2\}$, 
\begin{eqnarray}
\label{eqn:l2_conv}
\sqrt{\mathbb{E}\left[\lvert P^N_{(l_1,\,l_2)} - P\rvert^2\right]} \le C  (h_x^2 + h_y^2),
\end{eqnarray}
where $h_x = h_0\cdot 2^{-l_1},\,h_y = h_0\cdot 2^{-l_2}$, $N = T/k$.
\end{proposition}
\begin{proof}
Similar to Proposition 2.2 in \cite{reisinger2018analysis}.
\end{proof}

Let $\Delta_i$ be the first-order difference operator along directions $i=1,2$, defined as
\begin{equation}\label{eq_2dfirstorderdifference}
\Delta_iP_{\bm{l}}^N = 
\begin{cases}
P_{\bm{l}}^N-P_{\bm{l}-\bm{e}_i}^N,\quad &\text{if}\ l_i>0,\\
P_{\bm{l}}^N&\text{if}\ l_i=0,
\end{cases}
\end{equation}
with $\bm{e}_i$ the canonical unit vectors in $\mathbb{R}^2$, i.e., $(\bm{e}_i)_j = \delta_{ij}$, and $\bm{l}=(l_1,l_2)$.
We also define the first-order mixed difference operator $\Delta = \Delta_1\otimes\Delta_2$. Hence, for $l_1>0,l_2>0$,
\begin{equation}\label{eq_2dDeltaL}
\Delta P_{(l_1,\,l_2)}^N = P_{(l_1,\,l_2)}^N-P_{(l_1,\,l_2-1)}^N-P_{(l_1-1,\,l_2)}^N+P_{(l_1-1,\,l_2-1)}^N.
\end{equation}

We will prove in Section~\ref{sec_proofoflemmaSgrids} the following theorem.

\begin{theorem}\label{thm_sgridserror}
Assume \eqref{eq_stablerhos} holds. Let $h_x = h_{0}\cdot2^{-l_1},\ h_y = h_{0}\cdot2^{-l_2}$, and for any fixed $\lambda>0$, $k/\min\{h_x^2, h_y^2\} \leq \lambda$, $N = T/k$. Then for $\Delta P^N_{(l_1,\,l_2)}$ from \eqref{eq_2dDeltaL}, 
\begin{equation}
\label{eqn:delta_order}
\Big|\mathbb{E}\left[\Delta P^N_{(l_1,\,l_2)}\right]\Big| = O(h_x^2h_y^2),\qquad
\mathbb{E}\Big[\left|\Delta P^N_{(l_1,\,l_2)}\right|^2\Big]= O(h_x^4h_y^4).
\end{equation}
\end{theorem}
\begin{proof}
See Section \ref{sec_proofoflemmaSgrids}.
\end{proof}

\begin{remark}
We emphasise that \eqref{eqn:delta_order} does not follow from \eqref{eqn:l2_conv}, but ascertains essentially  that the difference operator cancels out any terms in the error expansion which depend on $h_x$ or $h_y$ alone. Establishing this property, which depends on the regularity of the problem, is the crucial step for the application of the sparse combination technique in the deterministic case of PDEs (see \cite{griebel1990combination, bungartz2004sparse, reisinger2012analysis}).
It is precisely the condition required for the Multi-Index Monte Carlo method of \cite{ref5}. Indeed, the approach here is interpretable as a specific multi-index decomposition applied
to the spatial variables, treating time separately.
\end{remark}


\begin{corollary}
With the same setting as in Theorem \ref{thm_sgridserror}, the first order differences of $P^N_{(l_1,l_2)}$ derived from \eqref{eq_2dloss} have the first and second moments
\begin{align*}
&\Big|\mathbb{E}\left[\Delta_1 P^N_{(l_1,\,l_2)}\right]\Big| = O(h_x^2),\qquad
\mathbb{E}\Big[\left|\Delta_1 P^N_{(l_1,\,l_2)}\right|^2\Big]= O(h_x^4), \\
&\Big|\mathbb{E}\left[\Delta_2 P^N_{(l_1,\,l_2)}\right]\Big| = O(h_y^2),\qquad
\mathbb{E}\Big[\left|\Delta_2 P^N_{(l_1,\,l_2)}\right|^2\Big]= O(h_y^4). 
\end{align*}
\end{corollary}

Following the ideas in \cite{bungartz2004sparse, ref5}, given a sequence of index sets $\mathcal{I}_l= \{(l_1,\,l_2)\in\mathbb{N}_0^2:\,l_1+l_2\leq l+1\}$, the approximation on level $l$ is defined as 
\begin{equation}\label{eq_sgridsapprox}
P_l^N = \sum_{(l_1,\,l_2)\in\mathcal{I}_l} \Delta P_{(l_1,\,l_2)}^N.
\end{equation}
Note that we use the same $k$ for all $\Delta P_{(l_1,\,l_2)}^N$, $(l_1,\,l_2)\in\mathcal{I}_l$. We have the following.

\begin{theorem}\label{thm_sgridsapprox}
Assume \eqref{eq_stablerhos} holds. 
Let 
$k\leq \lambda\, h_0^2\, 2^{-2l}$ for a fixed $\lambda>0$, $N = T/k$, $P$ given by \eqref{eq_2dLossTheoretical} 
and $P_l^N$ 
given by \eqref{eq_sgridsapprox}. Then 
\begin{equation}
\sqrt{\mathbb{E}\left[\lvert P^N_{l} - P\rvert^2\right]} = O(l\,2^{-2l}).
\end{equation}
\end{theorem}
\begin{proof}
See Section \ref{sec_proofofSgridsApprox}.
\end{proof}

%

The standard Monte Carlo estimator for $\mathbb{E}\big[P_l^N\big]$ using $M$ samples is defined as
\begin{equation}
\widehat{Y} = \frac{1}{M}\sum_{m=1}^{M}\sum_{(l_1,\,l_2)\in\mathcal{I}_l}\!\!\! \Big(\Delta \widehat{P}_{(l_1,\,l_2)}^{N}\Big)^{(m)},
\end{equation}
where $\Big(\Delta \widehat{P}_{(l_1,\,l_2)}^{N}\Big)^{(m)}$ is the $m$-th sample for the difference on spatial grid level $(l_1,\,l_2)$ of the SPDE approximation using $N$ time steps.

To reduce the bias below $\varepsilon$, we can choose $l = \big[\frac{1}{2}(-\log_2\varepsilon + \log_2|\log\varepsilon|)\big]$, where~$[\,\cdot\,]$ is the closest integer. Since $k\leq \lambda h_0^2 2^{-2l} = O(2^{-2l})$,
the computational cost for one sample of $P_{(l_1,l_2)}^N$ is
\[
W = \sum_{l_1+l_2\leq l+1} 2^{l_1+l_2}\cdot O(k^{-1}) = O(l\,2^{3l}) = O\big(\varepsilon^{-\frac{3}{2}}|\log\varepsilon|^{\frac{5}{2}}\big).
\]
Using standard Monte Carlo sampling, we need $M = O(\varepsilon^{-2})$ samples to reduce the variance below $\varepsilon^2$, hence the total computational cost to achieve a RMSE $\varepsilon$ is
\[
W = O\big(\varepsilon^{-\frac{7}{2}}|\log\varepsilon|^{\frac{5}{2}}\big).
\]

\subsection{Sparse combination MLMC estimators}\label{sec_2dMLMC}
Next, we combine the sparse grid combination technique with the MLMC method. We will show that this combination leads to the optimal complexity $O(\varepsilon^{-2})$.

Let $P_l\coloneqq P_l^{N_l}$ be an approximation to $P$ as in \eqref{eq_sgridsapprox} using a numerical discretisation with timestep $k_l$ and index set $\mathcal{I}_l$, where
\[
\mathcal{I}_l = \{(l_1,\,l_2)\in\mathbb{N}_0^2:\,l_1+l_2\leq l+1\},\qquad k_l = k_0\cdot 2^{-2l},\qquad N_l = \frac{T}{k_l}.
\]
For $(l_1,\,l_2)\in\mathcal{I}_l$, the mesh sizes are $h_x^{(l_1)} = h_{0}\cdot2^{-l_1}$, $h_y^{(l_2)} = h_{0}\cdot2^{-l_2}$. By 
using different levels of refinement, we have the following identity:
\[
\mathbb{E}[P] = \mathbb{E}[P_0] + \sum_{l=1}^\infty\mathbb{E}[P_l - P_{l-1}].
\]
Then we define the difference operators (acting on the level)
\begin{equation}\label{eq_2dexplicitDeltaL}
\delta P_l = P_l - P_{l-1} = \sum_{(l_1,\,l_2)\in\mathcal{I}_l} \Delta P_{(l_1,\,l_2)}^{N_l} - \sum_{(l_1,\,l_2)\in\mathcal{I}_{l-1}} \Delta P_{(l_1,\,l_2)}^{N_{l-1}},\quad l\geq 0,
\end{equation}
where we denote $P_{-1}\coloneqq 0$. Thus, the approximation $\mathbb{E}[P_{l^\ast}]$ to $\mathbb{E}[P]$ at level $l^\ast$ is
\[
\sum_{l=0}^{l^\ast}\mathbb{E}[\delta P_l] = \mathbb{E}\bigg[ \sum_{(l_1,\,l_2)\in\mathcal{I}_0} \Delta P_{(l_1,\,l_2)}^{N_0}\bigg] + \sum_{l=1}^{l^\ast}\mathbb{E}\bigg[\sum_{(l_1,\,l_2)\in\mathcal{I}_l} \Delta P_{(l_1,\,l_2)}^{N_l} - \sum_{(l_1,\,l_2)\in\mathcal{I}_{l-1}} \Delta P_{(l_1,\,l_2)}^{N_{l-1}}\bigg] .
\]
Therefore, instead of directly simulating $P_{l^\ast}$, we simulate $\delta P_l,\,l=0,1,\ldots,l^\ast$ separately. The key point is that we use the same Brownian path for $P_l$ and $P_{l-1}$ to calculate $\delta P_l = P_l - P_{l-1}$ such that the variance is considerably reduced. We have:
\begin{equation}
\label{eq_deltaLl}
\mathbb{E}[\delta P_l] \leq C_1\cdot l\,2^{-2l}, \quad
\mathrm{Var}[\delta P_l] \leq C_2\cdot l^2\,2^{-4l}, \quad
\mathrm{Cost}[\delta P_l] \leq C_3\cdot l\,2^{3l}, 
\end{equation}
where the first two inequalities follow from Theorem \ref{thm_sgridsapprox} and the third is immediate.

Let $\widehat{Y}_l$ be an estimator for $\mathbb{E}\big[\delta P_l\big]$ using $M_{l}$ independent samples $\delta \widehat{P}_{l}^{(m)}$ of $\delta P_l$,
\begin{equation*}
\widehat{Y}_{l} = \frac{1}{M_{l}}\sum_{m=1}^{M_{l}}\delta \widehat{P}_{l}^{(m)},\quad l=1,\ldots,l^\ast.
\end{equation*}
The MLMC estimator is then defined as $\widehat{P}_{l^\ast} \coloneqq \sum_{l=0}^{l^\ast}\widehat{Y}_{l}$ or
\begin{eqnarray*}
\widehat{P}_{l^\ast} 
&=&  \frac{1}{M_{0}}\sum_{m=1}^{M_{0}}\sum_{(l_1,\,l_2)\in\mathcal{I}_0}\!\!\! \Big(\Delta \widehat{P}_{(l_1,\,l_2)}^{N_0}\Big)^{(m)}  + \\
&& \qquad\qquad \sum_{l=1}^{l^\ast}\frac{1}{M_{l}}\sum_{m=1}^{M_{l}}\bigg[\sum_{(l_1,\,l_2)\in\mathcal{I}_l}\!\!\! \Big(\Delta \widehat{P}_{(l_1,\,l_2)}^{N_l}\Big)^{(m)} - \!\!\! \sum_{(l_1,\,l_2)\in\mathcal{I}_{l-1}}\!\!\! \Big(\Delta \widehat{P}_{(l_1,\,l_2)}^{N_{l-1}}\Big)^{(m)}\bigg],
\end{eqnarray*}
where $\Big(\Delta \widehat{P}_{(l_1,\,l_2)}^{N_l}\Big)^{(m)}$ is the $m$-th sample for the difference operator on spatial grid level $(l_1,\,l_2)$ using $N_l$ time steps. Following \cite[Theorem 3.1]{giles2008multilevel}, choosing $M_l$ to minimise the computational cost for a fixed variance, using \eqref{eq_deltaLl} and
noticing that the variance decreases strictly faster than the cost increases (the polynomial terms in $l$ being negligible),
we achieve the optimal complexity $O(\varepsilon^{-2})$.

\section{Fourier analysis of the sparse combination error expansion}\label{sec_2dMIMCEstimator}


We will prove Theorem~\ref{thm_sgridserror} and Theorem \ref{thm_sgridsapprox} in this section.
We employ a Fourier transform and then analyse the different wave number regions separately, a technique that has been successfully
used to derive error expansions for numerical approximations to PDEs (see \cite{ref3}) and SPDEs (see \cite{reisinger2018analysis}).

\subsection{Fourier transform of the solution}\label{sec_2dMIMCFourierTransform}

Define  the Fourier transform pair
\begin{align*}
\widetilde{v}(t,\xi,\eta) &= \int_{-\infty}^\infty\int_{-\infty}^\infty v(t,x,y)\mathrm{e}^{-\mathrm{i}\xi x -\mathrm{i}\eta y}\,\mathrm{d}x\,\mathrm{d}y,\\
v(t,x,y)&=\frac{1}{4\pi^2}\int_{-\infty}^\infty\int_{-\infty}^\infty\widetilde{v}(t,\xi,\eta)\mathrm{e}^{\mathrm{i}\xi x + \mathrm{i}\eta y}\,\mathrm{d}\xi\,\mathrm{d}\eta.
\end{align*}
The Fourier transform of \eqref{eq_SPDE} yields
\begin{equation}\label{eq_fourier}
\mathrm{d}\widetilde{v} = 
-\bigg(\big(\mathrm{i}\mu_x\xi + \mathrm{i}\mu_y\eta + \frac{\xi^2}{2} + \sqrt{\rho_x\rho_y}\rho_{xy}\xi\eta + \frac{\eta^2}{2} \big)\,\mathrm{d}t 
+ \mathrm{i}\sqrt{\rho_x}\xi\,\mathrm{d}W_t^x + \mathrm{i}\sqrt{\rho_y}\eta\,\mathrm{d}W_t^y \bigg)\widetilde{v},
\end{equation}
subject to the initial data $\widetilde{v}(0) = \mathrm{e}^{-\mathrm{i}\xi x_0 -\mathrm{i}\eta y_0}.$
For the remainder of the analysis, we take $\mu_x = \mu_y =0$. This does not alter the results (see Remark 2.3 in \cite{ref1} for 1d).

The solution to \eqref{eq_fourier} is
$$\widetilde{v}(t) = X(t)\mathrm{e}^{-\mathrm{i}\xi x_0 -\mathrm{i}\eta y_0},$$
where
\begin{equation}\label{eq_solXn}
X(t) = \exp\bigg(\!\!-\frac{1}{2}(1-\rho_x)\xi^2t -\frac{1}{2}(1-\rho_y)\eta^2t -\mathrm{i}\xi\sqrt{\rho_x}W_t^x - \mathrm{i}\eta\sqrt{\rho_y}W_t^y\bigg).
\end{equation}

For the numerical solution, we can use a discrete-continuous Fourier pair
\begin{eqnarray*}
V_{i,j}^0 &=& \frac{1}{4\pi^2 h_xh_y}\int_{-\pi}^{\pi}\int_{-\pi}^{\pi}
\widetilde{V}^0(\xi,\eta)\mathrm{e}^{\mathrm{i}\big((i-i_0)\xi + (j-j_0)\eta\big)}\,\mathrm{d}\xi\,\mathrm{d}\eta, \\
\widetilde{V}^0(\xi,\eta) &=& h_xh_y\sum_{i=-\infty}^\infty\sum_{j=-\infty}^\infty
\ V_{i,j}^0\mathrm{e}^{\mathrm{i}\big(-(i-i_0)\xi-(j-j_0)\eta\big)},
\end{eqnarray*}
where $i_0 = x_0/h_x$, $j_0 = y_0/h_y$.
It follows from \eqref{eq_2dDiracInitialApprox} that $\widetilde{V}^0(\xi,\eta) = 1.$ Then we have
\begin{equation}\label{eq_VijFourier}
\begin{aligned}
V_{i,j}^n &= 
 \frac{1}{4\pi^2 }\int_{-\frac{\pi}{h_y}}^{\frac{\pi}{h_y}}\int_{-\frac{\pi}{h_x}}^{\frac{\pi}{h_x}} 
\widetilde{V}^n(\xi,\eta)\mathrm{e}^{\mathrm{i}\big((i-i_0)\xi h_x + (j-j_0)\eta h_y\big)}\,\mathrm{d}\xi\,\mathrm{d}\eta,
\end{aligned}
\end{equation}
where we make the ansatz
\begin{equation}\label{eq_ansatzXn}
\widetilde{V}^n(\xi,\eta) = X_n(\xi,\eta)\widetilde{V}^0(\xi,\eta) = X_n(\xi,\eta).
\end{equation}

It follows from \eqref{eq_solXn} that the exact solution of $X(t_{n+1})$ given $X(t_n)$ is
\begin{equation*}
X(t_{n+1}) = X(t_n)\exp\bigg(\!\!-\frac{\xi^2}{2}(1-\rho_x) k -\frac{\eta^2}{2}(1-\rho_y)k -\mathrm{i}\xi\sqrt{\rho_xk}Z_{n,x} - \mathrm{i}\eta\sqrt{\rho_yk}\widetilde{Z}_{n,y}\bigg),
\end{equation*}
where $W_{t_{n+1}}^x-W_{t_n}^x =: \sqrt{k}Z_{n,x},\ W_{t_{n+1}}^y-W_{t_n}^y=: \sqrt{k}\widetilde{Z}_{n,y}$ are the Brownian increments.

Now we consider the numerical approximation of $X_n$. Let
\begin{align*}
\nonumber
&X_{n+1} = C_n\,X_n, \\
&C_n := \exp\bigg(\!\!-\frac{\xi^2}{2}(1-\rho_x)k -\frac{\eta^2}{2}(1-\rho_y)k -\mathrm{i}\xi\sqrt{\rho_xk}Z_{n,x} - \mathrm{i}\eta\sqrt{\rho_yk}\widetilde{Z}_{n,y} + e_n\bigg),
\end{align*}
and $e_n$ is the logarithmic error between the numerical solution and the exact solution introduced during $[nk,(n+1)k]$.
Aggregating over $N$ time steps, at $t_N = kN = T$, 
\begin{equation}\label{eq_XN}
X_N = \prod_{n=0}^{N-1} C_n = X(T)\exp\bigg(\sum_{n=0}^{N-1}e_n\bigg),
\end{equation}
where 
\begin{equation*}
X(T) = \exp\bigg(\!\!-\frac{\xi^2}{2}(1-\rho_x)T -\frac{\eta^2}{2}(1-\rho_y)T -\mathrm{i}\xi\sqrt{\rho_xk}\sum_{n=0}^{N-1}Z_{n,x} - \mathrm{i}\eta\sqrt{\rho_yk}\sum_{n=0}^{N-1}\widetilde{Z}_{n,y}\bigg)
\end{equation*}
is the exact solution at time $T$. 
Moreover, inserting $V_{i,j}^n$ from \eqref{eq_2DimplicitMilstein} to \eqref{eq_VijFourier}, we have
\begin{equation}\label{eq_Xn2}
\begin{aligned}
X_{n+1}(\xi,\eta) &= \frac{1}{1-(a_x+a_y)k}\Big(1 -\mathrm{i}c_x\sqrt{\rho_xk}Z_{n,x} -\mathrm{i}c_y\sqrt{\rho_yk}\widetilde{Z}_{n,y}\\
&\quad + b_x\rho_xk(Z_{n,x}^2-1) + b_y\rho_yk(\widetilde{Z}_{n,y}^2-1) + d\sqrt{\rho_x\rho_y}kZ_{n,x}\widetilde{Z}_{n,y}\Big)X_n(\xi,\eta),
\end{aligned}
\end{equation}
where
\begin{align*}
a_x &= -\frac{2\sin^2\frac{\xi h_x}{2}}{h_x^2},\quad
b_x = -\frac{\sin^2\xi h_x}{2h_x^2},\quad
c_x = \frac{\sin \xi h_x}{h_x},\quad
d = -\frac{\sin\xi h_x\sin\eta h_y}{h_xh_y},\\
a_y &= -\frac{2\sin^2\frac{\eta h_y}{2}}{h_y^2},\quad
b_y = -\frac{\sin^2\eta h_y}{2h_y^2},\quad
c_y = \frac{\sin \eta h_y}{h_y}.
\end{align*}

Hence, $C_n$ can also be expressed as
\begin{equation*}
\frac{1 \!-\!\mathrm{i} \sqrt{k} \big(c_x\sqrt{\rho_x}Z_{n,x} \!+\! c_y\sqrt{\rho_y}\widetilde{Z}_{n,y}\big) \!+\! 
k \big(b_x\rho_x(Z_{n,x}^2\!\!-\!1) + b_y\rho_y(\widetilde{Z}_{n,y}^2\!\!-\!1) \!+\! d\sqrt{\rho_x\rho_y}Z_{n,x}\widetilde{Z}_{n,y}\big)}{1-(a_x+a_y)k}.
\end{equation*}

\subsection{Fourier transform of the sparse combination estimator}\label{sec_sgridsestimator}

With the notation from Section \ref{sec_SparseGridsApprox}, omitting $N$ in $P$ to simplify the notation,
\begin{equation}
\begin{aligned}
P_{(l_1,\,l_2)} &= h_xh_y\sum_{i=1}^{\infty}\sum_{j=1}^{\infty} V_{i,j}^N + \frac{h_x}{2}\sum_{j=1}^{\infty}V_{0,j}^N + \frac{h_y}{2}\sum_{i=1}^{\infty}V_{i,0}^N + \frac{1}{4}h_xh_y V_{0,0}^N \\
&= \frac{1}{4\pi^2}\int_{-\frac{\pi}{h_y}}^{\frac{\pi}{h_y}}\int_{-\frac{\pi}{h_x}}^{\frac{\pi}{h_x}}  h_xh_yX_N^{(l_1,\,l_2)}(\xi,\eta)\chi(h_x,\xi)\chi(h_y,\eta)\,\mathrm{d}\xi\,\mathrm{d}\eta,
\end{aligned}
\end{equation}
where $\chi(h_x,\xi),\ \chi(h_y,\eta)$ is defined in a distributional sense as
\begin{equation}\label{eq_2dchi}
\chi(h_x,\xi) = \sum_{j=-i_0+1}^{\infty} \mathrm{e}^{\mathrm{i}jh_x\xi}+\frac{1}{2}\mathrm{e}^{-\mathrm{i}i_0h_x\xi},\qquad \chi(h_y,\eta) = \sum_{j=-j_0+1}^{\infty} \mathrm{e}^{\mathrm{i}jh_y\eta}+\frac{1}{2}\mathrm{e}^{-\mathrm{i}j_0h_y\eta}.
\end{equation}
Note that $\chi$ only appears multiplied by the smooth, fast decaying function $X_N$ and in integral form, such that this is well-defined.

We recall from \eqref{eq_2dDeltaL} the sparse combination estimator
\begin{equation}
\Delta P_{(l_1,\,l_2)}= P_{(l_1,\,l_2)}-P_{(l_1,\,l_2-1)}-P_{(l_1-1,\,l_2)}+P_{(l_1-1,\,l_2-1)}.
\end{equation}

We assume $h_x,\ h_y<1$. Even though $P_{(l_1,\,l_2)},\ P_{(l_1,\,l_2-1)},\ P_{(l_1-1,\,l_2)},\ P_{(l_1-1,\,l_2-1)}$ have different Fourier domains, we define $\Omega_{\text{low}}$ shared by all of them, for $0<p<1/2$,
\[
\Omega_{\text{low}} = \big\{(\xi,\eta):\vert\xi\vert\leq h_x^{-p}\text{ and } \vert\eta\vert\leq h_y^{-p}\big\}.
\]
Then we define $I_{(l_1,\,l_2)}$ as
\begin{align*}
I_{(l_1,\,l_2)} = \frac{1}{4\pi^2} \iint_{\Omega_{(l_1,\,l_2)}\setminus\Omega_{\text{low}}}\!\!\!\!\! h_xh_yX_N^{(l_1,\,l_2)}(\xi,\eta)\chi(h_x,\xi)\chi(h_y,\eta)\,\mathrm{d}\xi\,\mathrm{d}\eta,
\end{align*}
where $\Omega_{(l_1,\,l_2)}
=[-\pi/h_x,\pi/h_x]\times [-\pi/h_y,\pi/h_y]$. Then
\[
P_{(l_1,\,l_2)} = \frac{1}{4\pi^2} \iint_{\Omega_{\text{low}}}\!\!\!\!\! h_xh_yX_N^{(l_1,\,l_2)}(\xi,\eta)\chi(h_x,\xi)\chi(h_y,\eta)\,\mathrm{d}\xi\,\mathrm{d}\eta + I_{(l_1,\,l_2)},
\]
and
\begin{align*}
\Delta P_{(l_1,\,l_2)}&= P_{(l_1,\,l_2)}-P_{(l_1,\,l_2-1)}-P_{(l_1-1,\,l_2)}+P_{(l_1-1,\,l_2-1)}\\
= \; \frac{1}{4\pi^2}&\iint_{\Omega_{\text{low}}}\bigg( h_xh_y X_N^{(l_1,\,l_2)}\chi(h_x,\xi)\chi(h_y,\eta) - 2h_xh_y X_N^{(l_1-1,\,l_2)}\chi(2h_x,\xi)\chi(h_y,\eta)\\
 -\;  2h_xh_y &\!\; X_N^{(l_1,\,l_2-1)}\chi(h_x,\xi)\chi(2h_y,\eta) + 4h_xh_yX_{N}^{(l_1-1,\,l_2-1)}\chi(2h_x,\xi)\chi(2h_y,\eta)\bigg)\,\mathrm{d}\xi\,\mathrm{d}\eta\\
 + \; I_{(l_1,\,l_2)}&\! - I_{(l_1-1,\,l_2)} - I_{(l_1,\,l_2-1)} + I_{(l_1-1,\,l_2-1)}.
\end{align*}

Denote
\begin{equation}\label{eq_2dG}
\begin{aligned}
G_{(l_1,\,l_2)}(\xi,\eta)&\coloneqq  h_xh_y X_N^{(l_1,\,l_2)}\chi(h_x,\xi)\chi(h_y,\eta) - 2h_xh_y X_N^{(l_1-1,\,l_2)}\chi(2h_x,\xi)\chi(h_y,\eta)\\
- 2h_xh_y &X_N^{(l_1,\,l_2-1)}\chi(h_x,\xi)\chi(2h_y,\eta) + 4h_xh_yX_{N}^{(l_1-1,\,l_2-1)}\chi(2h_x,\xi)\chi(2h_y,\eta).
\end{aligned}
\end{equation}
Then we have
\begin{equation}\label{eq_2dG2}
\begin{aligned}
G_{(l_1,\,l_2)}(\xi,\eta)
&= 4h_xh_y\chi(2h_x,\xi)\chi(2h_y,\eta)\big(X_N^{l_1,\,l_2} - X_{N}^{l_1,\,l_2-1} \!\!- X_N^{l_1-1,\,l_2} + X_{N}^{l_1-1,\,l_2-1}\big)\\
 +\big(2h_xh_y&\chi(2h_x,\xi)\chi(h_y,\eta) - 4h_xh_y\chi(2h_x,\xi)\chi(2h_y,\eta)\big)\big(X_N^{l_1,\,l_2} - X_N^{l_1-1,\,l_2}\big)\\
 +\big(2h_xh_y&\chi(h_x,\xi)\chi(2h_y,\eta) - 4h_xh_y\chi(2h_x,\xi)\chi(2h_y,\eta)\big)\big(X_N^{l_1,\,l_2} - X_N^{l_1,\,l_2-1}\big)\\
 +\big(h_xh_y\chi&(h_x,\xi)\chi(h_y,\eta) - 2h_xh_y\chi(h_x,\xi)\chi(2h_y,\eta) - 2h_xh_y\chi(2h_x,\xi)\chi(h_y,\eta)\\
 + 4h_xh_y\chi&(2h_x,\xi)\chi(2h_y,\eta)\big)X_N^{l_1,\,l_2},
\end{aligned}
\end{equation}
and 
\[
\Delta P_{(l_1,\,l_2)} = \frac{1}{4\pi^2}\iint_{\Omega_{\text{low}}}G_{(l_1,\,l_2)}(\xi,\eta)\,\mathrm{d}\xi\mathrm{d}\eta + I_{(l_1,\,l_2)} - 
I_{(l_1-1,\,l_2)} - I_{(l_1,\,l_2-1)} + I_{(l_1-1,\,l_2-1)}.
\]

\subsection{Proof of Theorem \ref{thm_sgridserror}}\label{sec_proofoflemmaSgrids}

In this section, we give a proof of Theorem \ref{thm_sgridserror}. The splitting into different wave number regions is motivated by the analysis of
the heat equation in \cite{ref3}.
We further separate $\Omega_{(l_1,\,l_2)}\setminus\Omega_{\text{low}} = \Omega_{\text{mid}}\cup\Omega_{\text{high}}$. Hence $\Omega_{(l_1,\,l_2)}$ is divided into three regions,
\begin{align*}
\Omega_{\text{low}} &= \big\{(\xi,\eta):\vert\xi\vert\leq h_x^{-p}\text{ and } \vert\eta\vert\leq h_y^{-p}\big\}, \\
\Omega_{\text{mid}} &= \big\{(\xi,\eta):\vert\xi\vert\leq h_x^{-p}\text{ and } h_y^{-p}<\vert\eta\vert\leq \pi/h_y\big\} \; \cup \\
& \qquad\qquad\;\,  \big\{(\xi,\eta):h_x^{-p}<\vert\xi\vert\leq \pi/h_x\text{ and } \vert\eta\vert\leq h_y^{-p}\big\}, \\
\Omega_{\text{high}} &= \big\{(\xi,\eta): h_x^{-p}<\vert\xi\vert\leq \pi/h_x\text{ and } h_y^{-p}<\vert\eta\vert\leq \pi/h_y\big\}.
\end{align*}

We state three lemmas without proof, before giving the proof of Theorem \ref{thm_sgridserror}.

\begin{lemma}[Low wave region]\label{lem_2dSGridlowwave}
For $G_{(l_1,\,l_2)}(\xi,\eta)$ introduced in \eqref{eq_2dG}, there exists a constant~$C>0$, such that for any $l_1,l_2\in\mathbb{N}$,
\[
\iint_{\Omega_{\text{low}}}G_{(l_1,\,l_2)}(\xi,\eta)\,\mathrm{d}\xi\mathrm{d}\eta = h_x^2h_y^2\cdot R(T),
\]
where $R(T)$ is a random variable 
satisfying
$
|\mathbb{E}[R(T)]|\leq C,\; \mathbb{E}[|R(T)|^2]\leq C.
$
\end{lemma}
\begin{proof}
See Appendix \ref{app_lem3.1}.
\end{proof}

\begin{lemma}[Middle wave region]\label{lem_2dSGridmidwave}
For the middle wave region, we have,
\[
 \mathbb{E}\Bigg[\,\bigg|\iint_{\Omega_{\text{mid}}} \!\!\!\!\! h_xh_yX_N(\xi,\eta)\chi(h_x,\xi)\chi(h_y,\eta)\,\mathrm{d}\xi\,\mathrm{d}\eta\bigg|^2\,\Bigg] = 
 o(h_x^r) + o(h_y^r),  \qquad \forall r>0.
\]
\end{lemma}
\begin{proof}
See Appendix \ref{app_lem3.2}.
\end{proof}

\begin{lemma}[High wave region]\label{lem_2dSGridhighwave}
For the high wave region, we have
\[
 \mathbb{E}\Bigg[\,\bigg|\iint_{\Omega_{\text{high}}} \!\!\!\!\! h_xh_yX_N(\xi,\eta)\chi(h_x,\xi)\chi(h_y,\eta)\,\mathrm{d}\xi\,\mathrm{d}\eta\bigg|^2\,\Bigg] = o(h_x^rh_y^r),
 \qquad \forall r>0.
\]
\end{lemma}
\begin{proof}
See Appendix \ref{app_lem3.3}.
\end{proof}


\begin{proof}[Theorem \ref{thm_sgridserror}]
We have derived in Section \ref{sec_sgridsestimator} that 
\[
\Delta P_{(l_1,\,l_2)} = \frac{1}{4\pi^2}\iint_{\Omega_{\text{low}}}G_{(l_1,\,l_2)}(\xi,\eta)\,\mathrm{d}\xi\mathrm{d}\eta + I_{(l_1,\,l_2)} - I_{(l_1-1,\,l_2)} - I_{(l_1,\,l_2-1)} + I_{(l_1-1,\,l_2-1)}.
\]

Lemma~\ref{lem_2dSGridmidwave} and Lemma \ref{lem_2dSGridhighwave} give that
\begin{align*}
\mathbb{E}\Big[\big|I_{(l_1,\,l_2)}\big|^2\Big] &\leq 2\, 
 \mathbb{E}\Bigg[\,\bigg|\iint_{\Omega_{\text{mid}}} \!\!\!\!\! h_xh_yX_N(\xi,\eta)\chi(h_x,\xi)\chi(h_y,\eta)\,\mathrm{d}\xi\,\mathrm{d}\eta\bigg|^2\,\Bigg] \\
& + 2\,\mathbb{E}\Bigg[\,\bigg|\iint_{\Omega_{\text{high}}} \!\!\!\!\! h_xh_yX_N(\xi,\eta)\chi(h_x,\xi)\chi(h_y,\eta)\,\mathrm{d}\xi\,\mathrm{d}\eta\bigg|^2\,\Bigg]= o(h_x^r) + o(h_y^r)
\end{align*}
for all $r>0$.
Similar estimates hold for $I_{(l_1-1,\,l_2)}$, $I_{(l_1,\,l_2-1)}$, and $I_{(l_1-1,\,l_2-1)}$.
Combining this with Lemma \ref{lem_2dSGridlowwave}, the two inequalities in \eqref{eqn:delta_order} follow.
\end{proof}

\subsection{Proof of Theorem \ref{thm_sgridsapprox}}\label{sec_proofofSgridsApprox}
\begin{proof}[Theorem \ref{thm_sgridsapprox}]
Since from Proposition \ref{prop_2dErrorofLoss}, 
\[
\lim_{l_1,l_2\rightarrow \infty}\sqrt{\mathbb{E}\left[\lvert P^N_{(l_1,\,l_2)} - P\rvert^2\right]} = O(k) = O(2^{-2l}),
\]
we have
\begin{align*}
\quad\sqrt{\mathbb{E}\left[\lvert P^N_{l} - P\rvert^2\right]} &=  \sqrt{\mathbb{E}\bigg[\Big|\sum_{l_1+l_2>l+1}\Delta P_{(l_1,l_2)}\Big|^2\bigg]} + O(2^{-2l}) \\
& \leq  \sum_{l_1+l_2>l+1}\sqrt{\mathbb{E}\Big[\big|\Delta P_{(l_1,l_2)}\big|^2\Big]} + O(2^{-2l})\\
& \le C\cdot\bigg[\sum_{l_1=0}^{l}\sum_{l_2=l+1-l_1}^\infty 2^{-2l_1-2l_2} + \sum_{l_1=l+1}^{\infty}\sum_{l_2 = 0}^\infty 2^{-2l_1-2l_2}\bigg] + O(2^{-2l})\\
& = C\Big(\frac{l}{3}2^{-2l} + \frac{7}{9}2^{-2l}\Big) + O(2^{-2l}) = O(l\,2^{-2l}).
\end{align*}
\end{proof}

\section{Numerical tests}\label{sec_2dNumerical}

In this section, we test the theoretical convergence results empirically.

\subsection{Mean and variance of hierarchical increments} 

First, we verify numerically 
for $l_1,\,l_2>0$, $h_x = h_0\cdot2^{-l_1},\, h_y = h_0\cdot2^{-l_2}$, the result from Theorem~\ref{thm_sgridserror} that
\[
\Big|\mathbb{E}\left[\Delta P^N_{(l_1,\,l_2)}\right]\Big| = O(h_x^2h_y^2)\sim 2^{-2(l_1+l_2)},\qquad
\mathbb{E}\Big[\left|\Delta P^N_{(l_1,\,l_2)}\right|^2\Big]= O(h_x^4h_y^4)\sim 2^{-4(l_1+l_2)}.
\]

We choose parameters $h_0 = 1,\ T=1,\ x_0= y_0 = 2,\ \mu_x = \mu_y = 0.0809,\ \rho_x = \rho_y = 0.2,\  \rho_{xy} = 0.45$,
and truncate the domain to $(x,y)\in[-8,12]\times[-8,12]$.

Table~\ref{table_sgrids_level} shows $\log_2\big|\mathbb{E}[\Delta P^N_{(l_1,\,l_2)}]\big|$ and $\log_2\mathrm{Var}[\Delta P^N_{(l_1\,l_2)}]$ with fixed timestep $k = 4^{-3}$, and different levels of mesh refinement. We can see from the table that the mean decreases by around~2 going from level $l$ to $l+1$, and the variance decreases by approximately 4, consistent with the theoretical prediction. Figure~\ref{fig_sgrids_meanvarplot} depicts the corresponding contour plots. 

\begin{table}
\center
{\begin{tabular}{@{}c|c|cccc}
\backslashbox{$l_1$}{$l_2$} & \multicolumn{1}{c|}{0} & \multicolumn{1}{c}{1} & \multicolumn{1}{c}{2} & \multicolumn{1}{c}{3} & \multicolumn{1}{c}{4} \\
\hline
$0$ & -0.0819 & -7.28 & -9.04 & -11.00 & -12.98 \\
\hline
$1$ & -7.30 & -13.73 & -15.49 & -17.44 & -19.43 \\
$2$ & -9.05 & -15.48 & -17.24 & -19.19 & -21.17 \\
$3$ & -10.99 & -17.43 & -19.19 & -21.13 & -23.12 \\
$4$ & -12.98 & -19.42 & -21.17 & -23.13 & -25.11 \\
\end{tabular}}
\quad
{\begin{tabular}{@{}c|c|cccc}
\backslashbox{$l_1$}{$l_2$}& \multicolumn{1}{c|}{0}  & \multicolumn{1}{c}{1} & \multicolumn{1}{c}{2} & \multicolumn{1}{c}{3} & \multicolumn{1}{c}{4} \\
\hline
$0$ & -8.34 & -13.90 & -17.36 & -21.20 & -25.15 \\
\hline
$1$ & -13.91 & -25.29 & -29.05 & -33.03 & -36.99 \\
$2$ & -17.35 & -28.34 & -32.26 & -36.19 & -40.24 \\
$3$ & -21.19 & -32.30 & -36.02 & -39.95 & -43.98 \\
$4$ & -25.14 & -36.16 & -39.97 & -43.96 & -47.85 \\
\end{tabular}}
\caption{$\log_2\big|\mathbb{E}[\Delta P^N_{(l_1,\,l_2)}]\big|$ and $\log_2\mathrm{Var}[\Delta P^N_{(l_1\,l_2)}]$ with $h_x = 2^{-l_1},\,h_y = 2^{-l_2}$, $k = 4^{-3}$.}
\label{table_sgrids_level}
\end{table}

\begin{figure}
\centering
\subfloat
{
\includegraphics[width=0.48\columnwidth, height=0.48\columnwidth, trim={0 0 3.cm 0},clip]{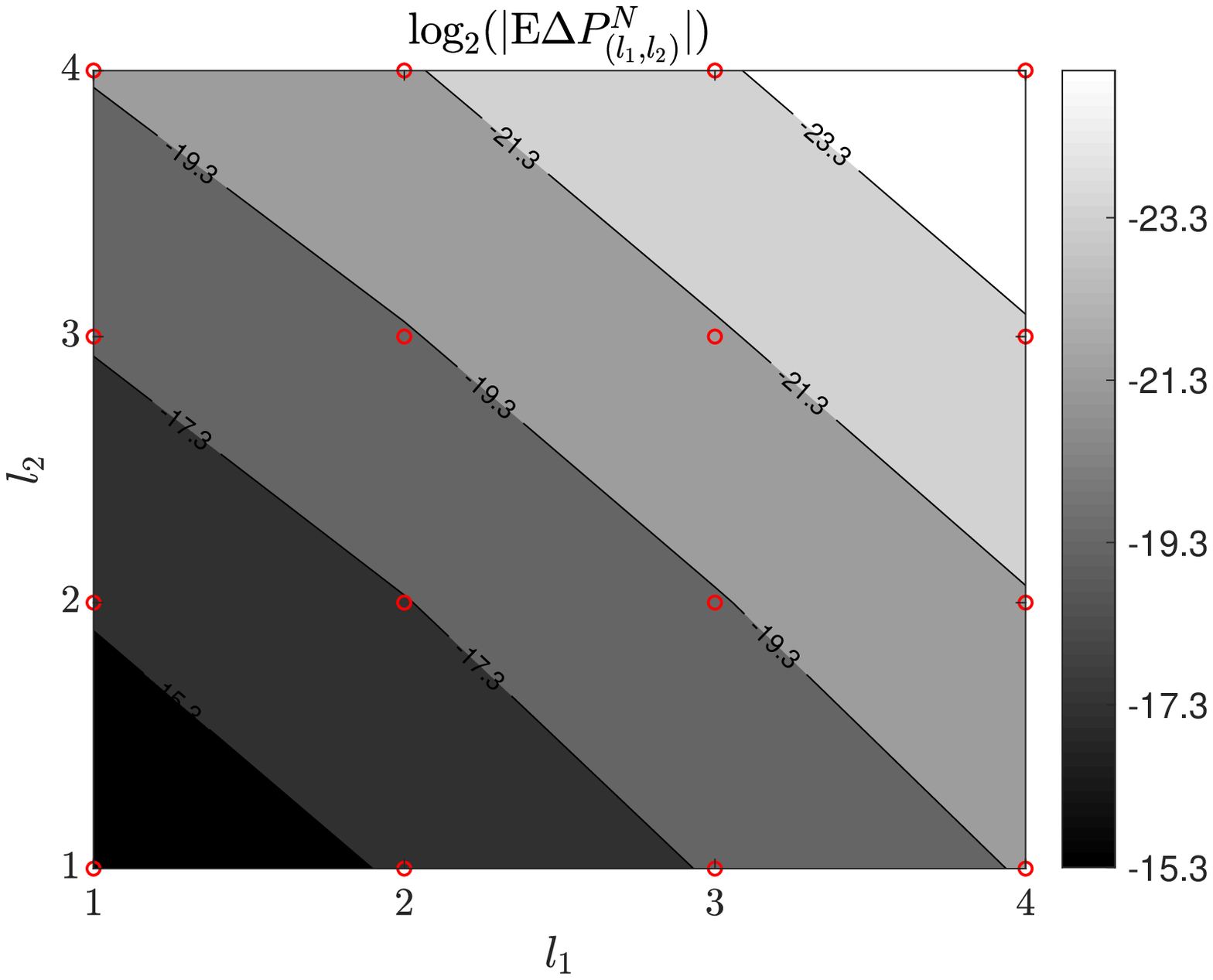}
\label{fig_sgrids_meanplot}
}
\subfloat
{
\includegraphics[width=0.48\columnwidth, height=0.48\columnwidth, trim={0 0 3.cm 0},clip]{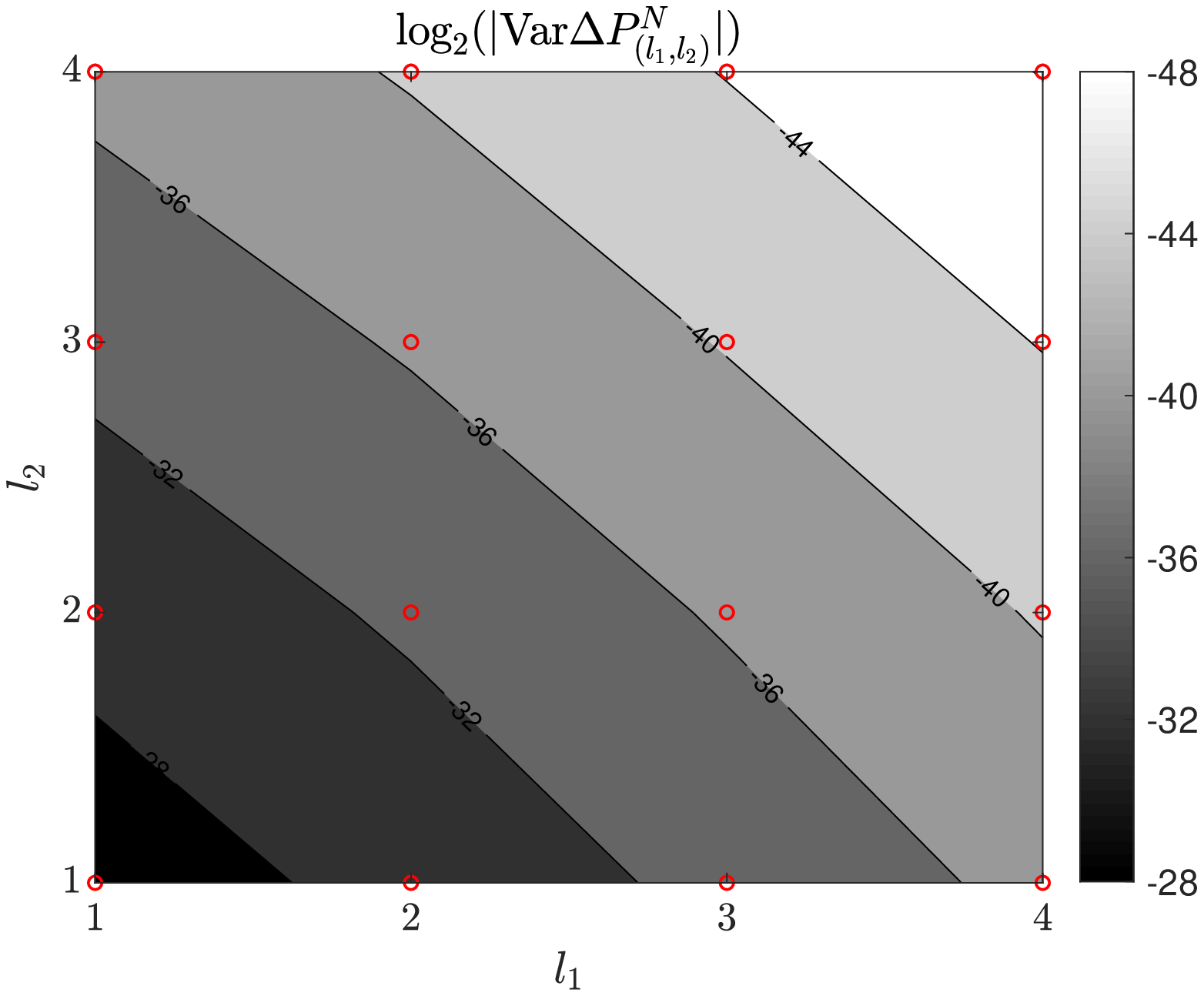}
\label{fig_sgrids_varplot}}
\caption{Contour plots of log values of sample mean and variance of $\Delta P^N_{(l_1\,l_2)}$ for $l_1,l_2>0$ used in the sparse combination method.}
\label{fig_sgrids_meanvarplot}
\end{figure}

\subsection{Mean and variance of sparse grid increments} 

Next, we estimate $\mathbb{E}\big[\delta P_l\big]$ and $\mathrm{Var}\big[\delta P_l\big]$, where $\delta P_l$ is given by \eqref{eq_2dexplicitDeltaL}. From Table~\ref{table_sgrids_level} we deduce that for the index set $I_l = \{(l_1,l_2):l_1+l_2\leq l+1\}$, the terms on the `boundary' (i.e., $l_1=0$ or $l_2=0$) will dominate. Although this does not affect the total order of complexity, we can further optimise the cost by modifying the index set such that the contribution from the boundary and `interior' are similar. 

Define therefore $\widehat{I}_l$ as the indices for interior meshes, and $\widetilde{I}_l$ for boundary meshes, 
\[
\widehat{I}_l = \{l_1>0,l_2>0:l_1+l_2\leq l\},\quad
\widetilde{I}_l = \{(l_1,0):0\leq l_1\leq l\}\cup \{(0,l_2):0\leq l_2\leq l\}.
\]
We balance the contributions from these two sets by finding, for some fixed $N$,
\[
l^\ast \!= \max\bigg\{\!\!\argmin_{l_x\geq 0} \big|\mathbb{E}\big[\Delta P^N_{(1,1)}\big]\big|/\big|\mathbb{E}\big[\Delta P^N_{(2+l_x,\,0)}\big]\big|, \ \argmin_{l_y\geq 0} \big|\mathbb{E}\big[\Delta P^N_{(1,1)}\big]\big| / \big|\mathbb{E}\big[\Delta P^N_{(0,2+l_y)}\big]\big|\!\!\bigg\},
\]
and using the index set
\[
\mathcal{I}_l= \begin{cases}
\widetilde{I}_l & \text{if } l<2+l^\ast,\\
\widetilde{I}_l \cup \widehat{I}_{l-l^\ast} & \text{if } l\geq 2+l^\ast.
\end{cases}
\] 

For example, from Table \ref{table_sgrids_level}, $l^\ast = 2$. So we have
\begin{align*}
&\mathcal{I}_0 = \{(0,0)\},\qquad \mathcal{I}_l = \mathcal{I}_{l-1} \cup \{(l,0),(0,l) \}, \quad l \in \{1,2,3\}, \\
& \mathcal{I}_4 = \mathcal{I}_3 \cup \{(4,0),(0,4),(1,1) \},\qquad \mathcal{I}_5 = \mathcal{I}_4 \cup \{(5,0),(0,5),(1,2),(2,1) \},\cdots
\end{align*}

We use $k_l = k_0\cdot2^{-2l}$ and $\mathcal{I}_l$ for the construction of $\delta P_l$. Table \ref{table_mlmcsgrids} verifies 
\[
\mathbb{E}[\delta P_l] = O(l\,2^{-2l}),\quad \mathrm{Var}[\delta P_l]= O(l^2\,2^{-4l}),\quad \mathrm{Cost}[\delta P_l]= O(l\,2^{3l}).
\]

\begin{table}
\center
\begin{tabular}{ccccccc}
\hline
\multicolumn{1}{c}{} & \multicolumn{1}{c}{$l=0$} & \multicolumn{1}{c}{$l=1$} & \multicolumn{1}{c}{$l=2$} & \multicolumn{1}{c}{$l=3$} & \multicolumn{1}{c}{$l=4$} & \multicolumn{1}{c}{$l=5$}  \\
\hline
$\log_2\big|\mathbb{E}\big[\delta P_l\big]\big|$ & -0.0647 & -7.62 & -9.44 & -11.42 & -13.30 & -15.22 \\
$\log_2\big(\mathrm{Var}\big[\delta P_l\big]\big)$ & -9.15 & -16.93 & -20.71 & -24.66 & -28.20 & -31.83 \\
$\log_2\big(\mathrm{Cost}[\delta P_l]\big)$ & 13.64 & 18.50 & 22.02 & 
25.22 & 28.44 & 31.65\\
\hline
\end{tabular}
\caption{Log values of mean, variance, cost for $\delta L_l$, $h_0 = 1/2,\ k_0 = 1/8$.}
\label{table_mlmcsgrids}
\end{table}

Figure \ref{fig_mlmcsgrids} are the corresponding plots.  The fitted slopes in Figure \ref{fig_mlmcsgrids} are $-1.91,\ -3.73$, and $3.27$, respectively,
compared to the theoretical asymptotic values of $-2$, $-4$, and~$3$ (neglecting logarithmic terms, which will play a role for low levels).

The cost is counted here as the total number of operations. Specifically, under the ADI scheme \eqref{eq_ADIdifference}, the cost of one numerical realisation of \eqref{eq_SPDE}
with mesh size $h_x,\,h_y$ and timestep $k$ is
\[
\mathrm{cost}/\mathrm{path} \ \sim \ \frac{x_{\text{max}} - x_{\text{min}}}{h_x}\cdot \frac{y_{\text{max}} - y_{\text{min}}}{h_y}\cdot \frac{T}{k}.
\]

\begin{figure}
\centering
\includegraphics[width=1.03 \columnwidth, trim={2.5cm 0 2.5cm 0.5 cm},clip]{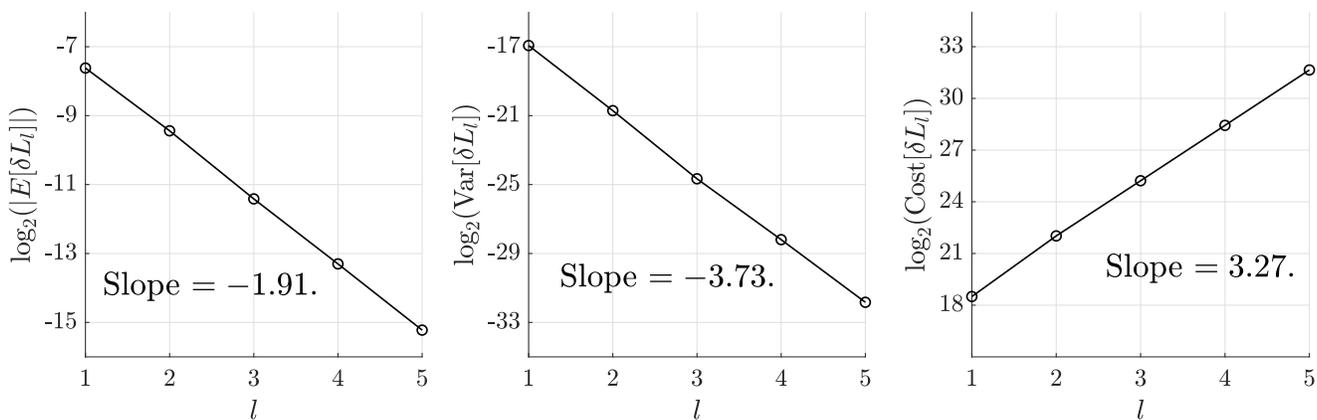}
\caption{Log values of mean, variance and cost for $\delta P_{l}$.}
\label{fig_mlmcsgrids}
\end{figure}

Since the mean square error for the estimator can be expressed as sum of the variance and the square of the weak error, we split the accuracy \enquote{budget} as
\begin{align}
\big|\mathbb{E}\big[P_l^N -P\big]\big| \leq \alpha\varepsilon, \qquad \mathrm{Var}\big[P_l^N \big] \leq (1-\alpha^2)\varepsilon^2,
\end{align}
and optimize over $\alpha$. Since
$V_{l}C_{l}\sim l^3\,2^{-l},$
the variance decays with levels more rapidly than the cost increases with levels, and thus the dominant cost is on level $0$. Therefore, we choose $\alpha$ relatively small to reduce the cost on level $0$, and hence we reduce the total cost. To find the optimal $\alpha$, one approach is to approximate the total cost given different $\alpha$, and choose the one which minimise the complexity.

\begin{figure}
\centering
\subfloat[MC, Sparse MC, MLMC, Sparse MLMC.]{
\includegraphics[width=0.48 \columnwidth, trim={0.8cm 0 1.5cm 0},clip]{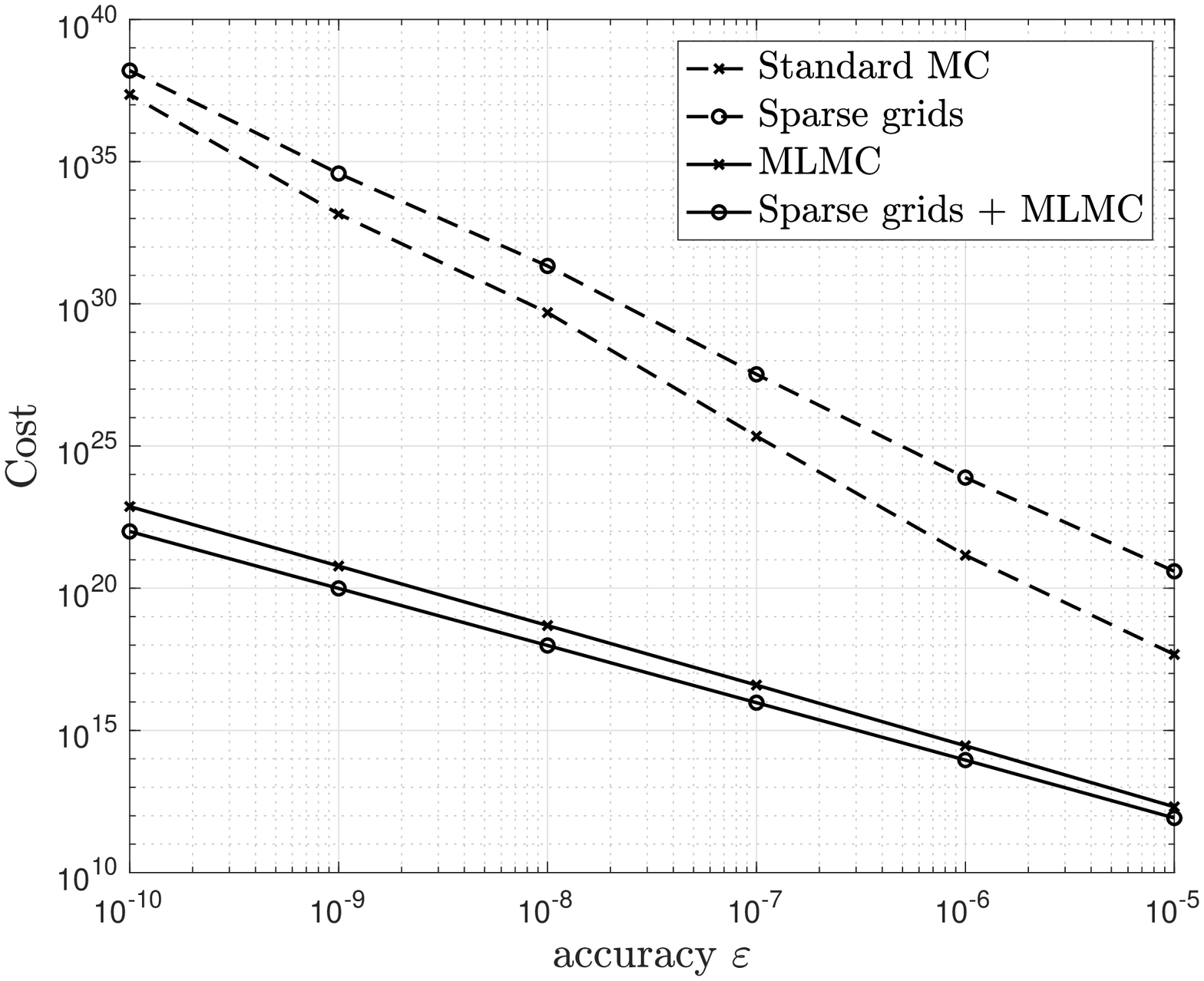}
\label{fig_comparecost_1}
}
\subfloat[Sparse and full grid MLMC.]{
\includegraphics[width=0.48 \columnwidth, trim={0.8cm 0 1.5cm 0},clip]{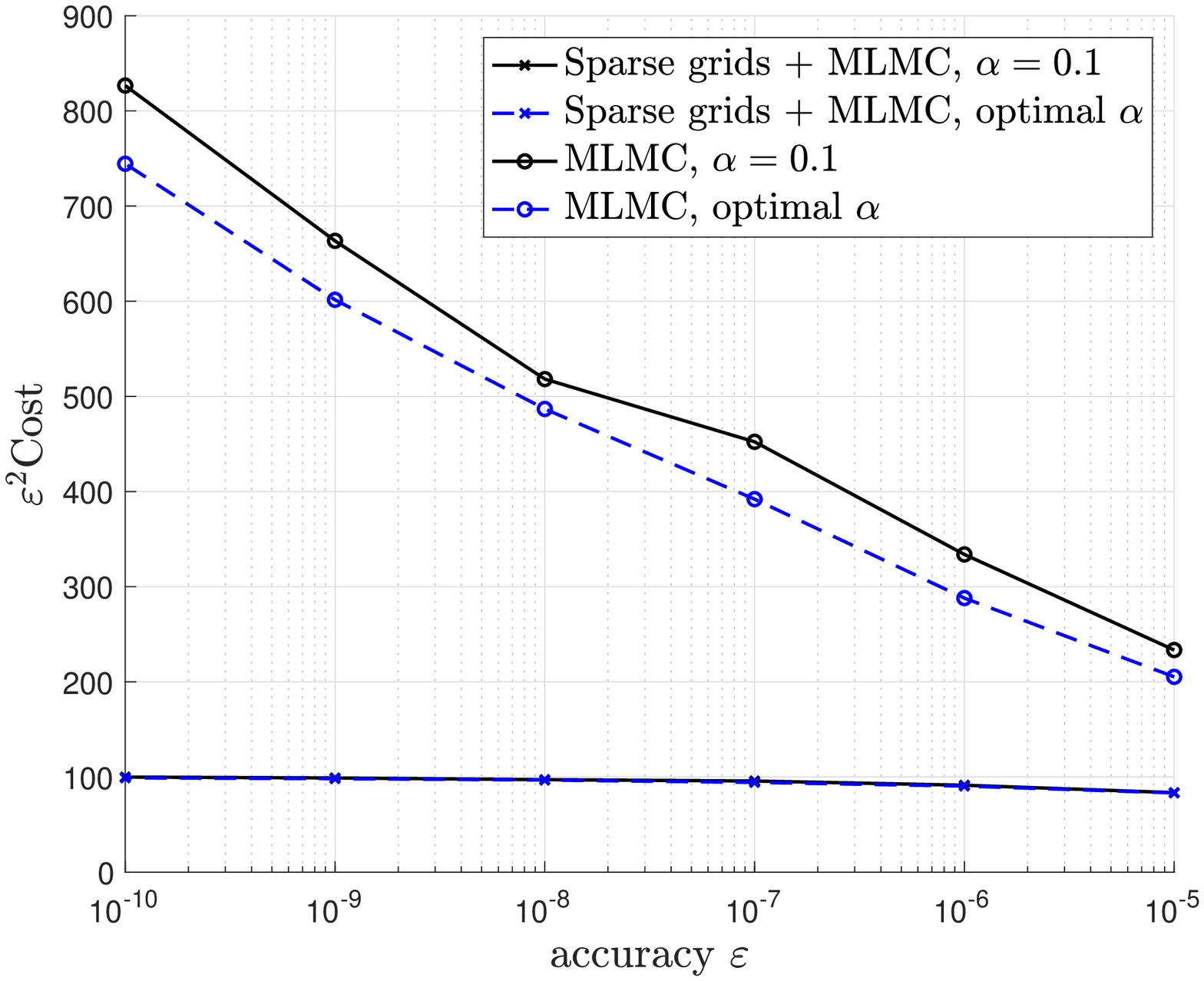}
\label{fig_comparecost_2}
}
\caption{Comparison of total cost among different schemes.}
\label{fig_comparecost}
\end{figure}

Figure~\ref{fig_comparecost_1} is the loglog plot of total cost among all the methods mentioned before: standard Monte Carlo, sparse combination, multilevel Monte Carlo, sparse combination with MLMC. All the schemes use the optimal~$\alpha$ for each accuracy. We can see from the graph that standard MC gives the cost $O(\varepsilon^{-4})$, and sparse combination gives $O(\varepsilon^{-7/2})$, as expected. As for MLMC and sparse combination with MLMC methods, both yield approximately $O(\varepsilon^{-2})$, which verifies our proof as the log term in MLMC cost is negligible in this plot.

Figure~\ref{fig_comparecost_2} compares $\varepsilon^2 W$ between sparse combination with MLMC and MLMC alone (without sparse combination). The total computational cost of sparse combination with MLMC is approximately proportional to $\varepsilon^{-2}$, hence $\varepsilon^2 W$ does not vary significantly for different accuracy $\varepsilon$. However, as the total cost of the multilevel scheme is proportional to $\varepsilon^{-2}(\log\varepsilon)^{2}$, and we can see from the figure that $\varepsilon^2 W$ increases as $\varepsilon$ goes to zero. For the black line, we use $\alpha = 0.1$. For the blue dotted line, each scheme uses the optimal $\alpha$ for different accuracies. As a result, sparse combination with MLMC achieves the optimal order of complexity.

\section{Generalisation to higher dimensions}\label{sec_highdim}

We can generalise the SPDE to $d$ dimensions, where $d\geq 3$. Let $(\Omega,\mathcal{F},\mathbb{P})$ be a probability space, where there is given a $d$-dimensional standard Brownian motion $W$ with correlation matrix $\Sigma = (\rho_{k,m})$. The natural extension of the SPDE \eqref{eq_SPDE} is
\begin{equation}\label{eq_ndspde}
\mathrm{d}v = -\sum_{k=1}^d \mu_k\frac{\partial v}{\partial x_k}\,\mathrm{d}t + \frac{1}{2}\sum_{k,m=1}^{d}\sqrt{\rho_k\rho_m}\,\rho_{k,m}\frac{\partial^2 v}{\partial x_k\partial x_m}\,\mathrm{d}t - \sum_{k=1}^d \sqrt{\rho{}_k}\frac{\partial v}{\partial x_k}\,\mathrm{d}W_t^k,
\end{equation}
for $x\in\mathbb{R}^d,\ 0<t\leq T$, where $\mu\in\mathbb{R}^d$ and $\rho\in(0,1)^d$ are parameters, subject to the Dirac initial datum 
$v(0,x) = \delta(x_1-x_{1,0}) \otimes\cdots\otimes \delta(x_d-x_{d,0})$,
where $x_0\in\mathbb{R}_+^d$ is given.


We use a spatial grid with uniform spacing $h_{x_i}>0$, $i=1,\ldots, d$, and let $V_{i_1,\ldots,i_d}^{m}$ be the approximation to $v(mk,i_1h_{x_1},\ldots,i_dh_{x_d})$, $m=1,\ldots,N$, $i_j\in\mathbb{Z}$, where $i_{0,j}\coloneqq[x_{0,j}/h_{x_j}]$, the closest integers to $x_{0,j}/h_{x_j}$.
We approximate $v(0,x_1,\ldots,x_d)$ by
\[
V_{i_1,\dots,i_d}^0 = \prod_{j=1}^d h_{x_j}^{-1}\delta_{(i_{0,1},\ldots,i_{0,d})} = 
\begin{cases}
\prod_{j=1}^d h_{x_j}^{-1},\quad &i_j=i_{0,j},\quad\forall j\in\{1,\ldots,d\},\\
0,&\text{otherwise}.
\end{cases}
\]

In analogy with \eqref{eq_2dlossfunctional}, we define a linear functional by
\begin{equation}\label{eq_ndloss}
P_T = \int_{\mathbb{R}_+^d} v(T,x)\,\mathrm{d}x.
\end{equation}

Similar as before, we use an implicit finite difference scheme to approximate \eqref{eq_ndspde}, and we use the trapezoidal rule for $P$ with a truncation of the domain. We have the following conjectures and results.

\begin{conjecture}
\label{conj_1}
Assume the implicit finite difference scheme is stable
\footnote{We expect this to be true for `small enough' correlations as in the two-dimensional case for \eqref{eq_stablerhos},
but it is not obvious what the conditions will be for specific $d>2$ without performing the analysis.},
and the timestep $k$ and mesh size $h_{x_i}$ satisfy 
\begin{equation}
\label{general_cfl}
k\leq \lambda\min\{h_{x_1}^2,h_{x_2}^2,\cdots,h_{x_d}^2\},
\end{equation}
for arbitrary fixed $\lambda>0$. Then we have the error expansion
\[
\sqrt{\mathbb{E}\big[| V_{i_1,\dots,i_d}^N - v(T,x_1,\ldots,x_d)|^2\big]}= 
O(h_{x_1}^2) + O(h_{x_2}^2) +\ldots + O(h_{x_d}^2 ),
\]
where $N=T/k$ is the number of time steps.

\end{conjecture}

Next we apply the sparse combination method to \eqref{eq_ndloss} with fixed $k$ satisfying (\ref{general_cfl}).
Similar to Section \ref{sec_SparseGridsApprox}, let $\Delta_i$ be the first-order difference operator along directions $i=1,\ldots, d$, defined as in \eqref{eq_2dfirstorderdifference}, and 
$\Delta = \Delta_1\otimes\cdots\otimes\Delta_d$.

\begin{conjecture}
\label{conj_2}
Assume the implicit finite difference scheme is stable. Let $h_{x_i} = h_{0}\cdot2^{-l_i}$, $i=1,\cdots,d$, and $k$ be the timestep such that for an arbitrary fixed $\lambda>0$, 
$k\leq \lambda\min\{h_{x_1}^2,h_{x_2}^2,\cdots,h_{x_d}^2\}$, where
$N = T/k$ the number of time steps. Then the first and second moments of $\Delta P^N_{\bm{l}}$ satisfy
\begin{equation}
\label{eq:bounds_d}
\Big|\mathbb{E}\left[\Delta P^N_{\bm{l}}\right]\Big| = O(h_{x_1}^2\cdots h_{x_d}^2),\qquad
\mathbb{E}\Big[\left|\Delta P^N_{\bm{l}}\right|^2\Big]= O(h_{x_1}^4\cdots h_{x_d}^4).
\end{equation}
\end{conjecture}

Given a sequence of index sets
\[
\mathcal{I}_l= \{\bm{l}\in\mathbb{N}_0^d:\,l_1+\cdots + l_d\leq l+1\},
\] 
the approximation on level $l$ is defined (similar to \eqref{eq_sgridsapprox}) as 
\begin{equation}
\label{PlN_gen}
{P}_l^N = \sum_{\bm{l}\in\mathcal{I}_l} \Delta P_{\bm{l}}^N.
\end{equation}
Note that we use the same $k$ for all $\Delta P^N_{\bm{l}}$, $(l_1,\cdots,l_d)\in\mathcal{I}_l$. Then we have:

\begin{proposition}\label{conj_errorL}
Assume Conjectures \ref{conj_1} and \ref{conj_2} to be true.
Then, for $P$ given by \eqref{eq_ndloss} 
and $P_l^N$ by \eqref{PlN_gen}, we have
\[
\sqrt{\mathbb{E}\left[\lvert P^N_{l} - P\rvert^2\right]} = O(l^{d-1}\,2^{-2l}).
\]
Moreover, {by choosing $l = \big[\frac{1}{2}(-\log_2\varepsilon + (d-1)\log_2|\log\varepsilon|)\big]$,} the computational cost $W$ to achieve a RMSE $\varepsilon$ using standard Monte Carlo estimation is 
{
\begin{equation}\label{eq_ndcostsgrid}
W = O\big(\varepsilon^{-\frac{7}{2}}|\log\varepsilon|^{\frac{5}{2}(d-1)}\big).
\end{equation}
}
\end{proposition}

Next, we combine MLMC with the sparse combination scheme. Similar to Section~\ref{sec_2dMLMC}, let $P_l\coloneqq P_l^{N_l}$ be an approximation to $P$ as in \eqref{eq_sgridsapprox} using a discretisation with timestep $k_l$ and index set $\mathcal{I}_l$ defined above. Then we define 
\begin{equation}\label{eq_nddeltaL}
\delta P_l = P_l - P_{l-1},\quad l\geq 0,
\end{equation}
where we denote $P_{-1}\coloneqq 0$. Thus the approximation to $P$ at level $l^\ast$ has the form
\[
\mathbb{E}[P_{l^\ast}] = \sum_{l=0}^{l^\ast}\mathbb{E}[\delta P_l] = \sum_{l=0}^{l^\ast}\mathbb{E}\bigg[\sum_{\bm{l}\in\mathcal{I}_l} \Delta P_{\bm{l}}^{N_l} - \sum_{\bm{l}\in\mathcal{I}_{l-1}} \Delta P_{\bm{l}}^{N_{l-1}}\bigg] .
\]

In this way, we simulate $\delta P_l,\,l=0,1,\ldots,l^\ast$ instead of directly simulating $P_{l^\ast}$, such that the variance of $\delta P_l = P_l - P_{l-1}$ is considerably reduced by using the same Brownian path for $P_l$ and $P_{l-1}$.

Let $\widehat{Y}_l$ be an estimator for $\mathbb{E}\big[\delta P_l\big]$ using $M_{l}$ samples. Each estimator is an average of $M_{l}$ independent samples, where $\delta \widehat{P}_{l}^{(m)}$ is the $m$-th sample arising from a single SPDE approximation,
\[
\widehat{Y}_{l} = \frac{1}{M_{l}}\sum_{m=1}^{M_{l}}\delta \widehat{P}_{l}^{(m)},\quad l=0,\ldots,l^\ast.
\]

The MLMC estimator is defined as
\begin{equation}\label{eq_ndMLMCL}
\widehat{P}_{l^\ast}\coloneqq\sum_{l=0}^{l^\ast}\widehat{Y}_{l} = \frac{1}{M_{l}}\sum_{m=1}^{M_{l}}\bigg[\sum_{\bm{l}\in\mathcal{I}_l} \Big(\Delta \widehat{P}_{\bm{l}}^{N_l}\Big)^{(m)} - \sum_{\bm{l}\in\mathcal{I}_{l-1}} \Big(\Delta \widehat{P}_{\bm{l}}^{N_{l-1}}\Big)^{(m)}\bigg],
\end{equation}
where $\Big(\Delta \widehat{P}_{\bm{l}}^{N_l}\Big)^{(m)}$ is the $m$-th sample for the difference on spatial grid level $\bm{l} = (l_1,\ldots,l_d)$ of the SPDE approximation using $N_l$ time steps. Following \cite{giles2008multilevel}, through optimising $M_l$ to minimise the computational cost for a fixed variance, we can achieve the optimal complexity $O(\varepsilon^{-2})$ in this case.

\begin{proposition}
Assume Conjectures \ref{conj_1} and \ref{conj_2} to be true.
Given $\delta P_l$ from \eqref{eq_nddeltaL}, there exist $C_1,\,C_2,\,C_3 >0$, such that
\begin{align*}
\mathbb{E}[\delta P_l] \leq C_1\cdot l^{d-1} \,2^{-2l}, \qquad
\mathrm{Var}[\delta P_l] \leq C_2\cdot l^{2d-2 }\,2^{-4l}, \qquad
\mathrm{Cost}[\delta P_l] \leq C_3\cdot l^{d-1} \,2^{3l}.
\end{align*}
Then, given a RMSE $\varepsilon$, the MLMC estimator \eqref{eq_ndMLMCL} leads to a total complexity $O(\varepsilon^{-2})$.
\end{proposition}
\begin{proof}
The first inequalities follow directly from \eqref{eq:bounds_d}.
The complexity then is a small modification of \cite[Theorem 3.1]{giles2008multilevel}.
\end{proof}

If we use MLMC without sparse combination to estimate $\mathbb{E}[P_t]$, by letting $h_{x_j} = h_0\cdot 2^{-l}$, $k = k_0\cdot 2^{-2l}$, then for constants $C_1,C_2,C_3>0$, independent of $h$ and $k$,
\[
E_l \le C_1\cdot 2^{-2l},\quad V_l \le C_2\cdot2^{-4l},\quad W_l \le C_3\cdot2^{(d+2)l}.
\]

Similarly, we have the following result.
\begin{proposition}
Given a RMSE $\varepsilon$, the total computational cost $W$ for $P$ in \eqref{eq_ndloss} using MLMC satisfies
\begin{equation}
\label{eq_ndcostmlmc}
W = O\big(\varepsilon^{-1-\frac{d}{2}}\big),\quad d\geq 3.
\end{equation}
\end{proposition}

Comparing \eqref{eq_ndcostmlmc} with \eqref{eq_ndcostsgrid}, we can see that when $d>5$, the sparse combination scheme with standard MC performs better than MLMC on regular grids.

\section{Conclusion}\label{sec_Conclusion}
We considered a two-dimensional parabolic SPDE and a functional of the solution. We analysed the accuracy and complexity of a sparse combination multilevel Monte Carlo estimator, and showed that, by using a semi-implicit Milstein finite difference discretisation \eqref{eq_2DimplicitMilstein},
we achieved the order of complexity $O(\varepsilon^{-2})$ for a RMSE $\varepsilon$, whereas the cost using standard Monte Carlo is $O(\varepsilon^{-4})$, and that using MLMC is $O(\varepsilon^{-2}(\log \varepsilon)^2)$. When generalising to higher-dimensional problems, sparse combination with MLMC maintains the optimal complexity, whereas MLMC has an increasing total cost as the dimension increases.

Further research will apply this method to a Zakai type SPDE with non-constant coefficients. Another open question is a complete analysis of the numerical approximation of initial-boundary value problems for the considered SPDE.

\appendix

\section{Proofs of some auxiliary stability results}
\label{app:proofs}
\subsection{Proof of Lemma \ref{lem_2dSGridlowwave}}\label{app_lem3.1}
\begin{proof}
For $(\xi,\eta)\in\Omega_{\text{low}}$, we have $G_{(l_1,\,l_2)}(\xi,\eta)$ as in \eqref{eq_2dG2}. From \eqref{eq_XN}, we have
\begin{equation}\label{eq_X1-X2}
\begin{aligned}
X_N^{l_1,\,l_2}-X_{N}^{l_1,\,l_2-1} &= X(T)\bigg(\exp\big(\sum_{n=0}^{N-1}e_n^{l_1,l_2}\big) - \exp\big(\sum_{n=0}^{N-1}e_n^{l_1,l_2-1}\big)\bigg),\\
X_N^{l_1,\,l_2}-X_{N}^{l_1-1,\,l_2} &= X(T)\bigg(\exp\big(\sum_{n=0}^{N-1}e_n^{l_1,l_2}\big) - \exp\big(\sum_{n=0}^{N-1}e_n^{l_1-1,l_2}\big)\bigg),\\
X_N^{l_1,\,l_2} - X_{N}^{l_1,\,l_2-1} - X_N^{l_1-1,\,l_2} + X_{N}^{l_1-1,\,l_2-1} &=X(T)\bigg(\exp\big(\sum_{n=0}^{N-1}e_n^{l_1,l_2}\big) - \exp\big(\sum_{n=0}^{N-1}e_n^{l_1,l_2-1}\big) \\
&\quad - \exp\big(\sum_{n=0}^{N-1}e_n^{l_1-1,l_2}\big) + \exp\big(\sum_{n=0}^{N-1}e_n^{l_1-1,l_2-1}\big)\bigg),\\
\end{aligned}
\end{equation}
where
\[
X(T) = \exp\bigg(-\frac{1}{2}(1-\rho_x)\xi^2T -\frac{1}{2}(1-\rho_y)\eta^2T -\mathrm{i}\xi\sqrt{\rho_xk}\sum_{n=0}^{N-1}Z_{n,x} - \mathrm{i}\eta\sqrt{\rho_yk}\sum_{n=0}^{N-1}\widetilde{Z}_{n,y}\bigg).
\]

By a similar proof to Lemma 4.1 in \cite{reisinger2018stability}, we have
\begin{align*}
X_N^{l_1,\,l_2} &= X(T)\cdot R_0(T,h_x\xi,h_y\eta),\\
X_N^{l_1,\,l_2} - X_{N}^{l_1,\,l_2-1} &= O(h_y^2)\cdot X(T)\cdot R_1(T,h_x\xi,h_y\eta),\\
X_N^{l_1,\,l_2} - X_{N}^{l_1-1,\,l_2-1} &= O(h_x^2)\cdot X(T)\cdot R_2(T,h_x\xi,h_y\eta),\\
X_N^{l_1,\,l_2} - X_{N}^{l_1,\,l_2-1} - X_N^{l_1-1,\,l_2} + X_{N}^{l_1-1,\,l_2-1} &= O(h_x^2h_y^2)\cdot X(T)\cdot R_3(T,h_x\xi,h_y\eta).
\end{align*}
Here, $R_0,\ R_1$, $R_2$, $R_3$ are random variables satisfying
\begin{align*}
\mathbb{E}\Big[ \big|X(T)\cdot R_i(T,h_x\xi,h_y\eta)\big|^n \Big] &= f_i(\xi,\eta)\mathrm{e}^{-\frac{n}{2}(\xi^2+\eta^2+2\xi\eta\sqrt{\rho_x\rho_y}\rho_{xy})T},\qquad i=0,1,2,3,
\end{align*}
where $f_i(\cdot)$ are polynomial functions. So in the low wave region, we get
\begin{align*}
&\quad\mathbb{E}\left[\iint_{\Omega_{\text{low}}}\!\!\!\!G_{(l_1,\,l_2)}(\xi,\eta)\,\mathrm{d}\xi\mathrm{d}\eta\right] \\
&= O(h_x^2h_y^2)\,\mathbb{E}\bigg[\iint_{\Omega_{\text{low}}}\!\!\! 4h_xh_y\chi(2h_x,\xi)\chi(2h_y,\eta)\Big(X(T)\, R_3(T,h_x\xi,h_y\eta)\Big)\mathrm{d}\xi\,\mathrm{d}\eta\bigg]\\
&\quad + O(h_x^2)\,\mathbb{E}\bigg[\iint_{\Omega_{\text{low}}}\!\!\! \Big(2h_xh_y\chi(2h_x,\xi)\chi(h_y,\eta) - 4h_xh_y\chi(2h_x,\xi)\chi(2h_y,\eta)\Big)\Big(X(T)\, R_2(T,h_x\xi,h_y\eta)\Big)\,\mathrm{d}\xi\,\mathrm{d}\eta\bigg]\\
&\quad + O(h_y^2)\,\mathbb{E}\bigg[\iint_{\Omega_{\text{low}}}\!\!\! \Big(2h_xh_y\chi(h_x,\xi)\chi(2h_y,\eta) - 4h_xh_y\chi(2h_x,\xi)\chi(2h_y,\eta)\Big)\Big(X(T)\, R_1(T,h_x\xi,h_y\eta)\Big)\,\mathrm{d}\xi\,\mathrm{d}\eta\bigg]\\
&\qquad\qquad + \mathbb{E}\bigg[\iint_{\Omega_{\text{low}}}\!\!\! \Big(h_xh_y\chi(h_x,\xi)\chi(h_y,\eta) - 2h_xh_y\chi(h_x,\xi)\chi(2h_y,\eta) - 2h_xh_y\chi(2h_x,\xi)\chi(h_y,\eta)\\
&\qquad\qquad\qquad\qquad\qquad + 4h_xh_y\chi(2h_x,\xi)\chi(2h_y,\eta)\Big)\Big(X(T)\, R_0(T,h_x\xi,h_y\eta)\Big)\,\mathrm{d}\xi\,\mathrm{d}\eta\bigg].
\end{align*}
Here,
\[
\iint_{\Omega_{\text{low}}} 4h_xh_y\chi(2h_x,\xi)\chi(2h_y,\eta)\Big(X(T)\cdot R_3(T,h_x\xi,h_y\eta)\Big)\,\mathrm{d}\xi\,\mathrm{d}\eta
\]
is the numerical approximation to 
\[
\int_0^\infty\int_0^\infty\bigg(\iint_{\Omega_\text{low}}\mathrm{e}^{\mathrm{i}\xi(x-x_0)+\mathrm{i}\eta(y-y_0)}\Big(X(T)\cdot R_3(T,h_x\xi,h_y\eta)\Big)\,\mathrm{d}\xi\,\mathrm{d}\eta\bigg) \,\mathrm{d}x\,\mathrm{d}y
\]
by the trapezoidal rule with error $O(h_x^2)+O(h_y^2)$. Therefore,
\[
\mathbb{E}\bigg[\iint_{\Omega_{\text{low}}} 4h_xh_y\chi(2h_x,\xi)\chi(2h_y,\eta)\Big(X(T)\cdot R_3(T,h_x\xi,h_y\eta)\Big)\,\mathrm{d}\xi\,\mathrm{d}\eta\bigg] = O(1).
\]
By the same reason,
\begin{align*}
&\mathbb{E}\bigg[\iint_{\Omega_{\text{low}}} \Big(2h_xh_y\chi(2h_x,\xi)\chi(h_y,\eta) - 4h_xh_y\chi(2h_x,\xi)\chi(2h_y,\eta)\Big)\Big(X(T)\cdot R_2(T,h_x\xi,h_y\eta)\Big)\,\mathrm{d}\xi\,\mathrm{d}\eta\bigg] = O(h_y^2),\\
&\mathbb{E}\bigg[\iint_{\Omega_{\text{low}}} \Big(2h_xh_y\chi(h_x,\xi)\chi(2h_y,\eta) - 4h_xh_y\chi(2h_x,\xi)\chi(2h_y,\eta)\Big)\Big(X(T)\cdot R_1(T,h_x\xi,h_y\eta)\Big)\,\mathrm{d}\xi\,\mathrm{d}\eta\bigg] = O(h_x^2),\\
&\mathbb{E}\bigg[\iint_{\Omega_{\text{low}}} \Big(h_xh_y\chi(h_x,\xi)\chi(h_y,\eta) - 2h_xh_y\chi(h_x,\xi)\chi(2h_y,\eta) - 2h_xh_y\chi(2h_x,\xi)\chi(h_y,\eta)\\
&\qquad\qquad\qquad\qquad\qquad\qquad\qquad\quad + 4h_xh_y\chi(2h_x,\xi)\chi(2h_y,\eta) \Big)\Big(X(T)\cdot R_0(T,h_x\xi,h_y\eta)\Big)\,\mathrm{d}\xi\,\mathrm{d}\eta\bigg] = O(h_x^2h_y^2).
\end{align*}

Therefore, we get 
\[
\mathbb{E}\left[\iint_{\Omega_{\text{low}}}\!\!\!\!G_{(l_1,\,l_2)}(\xi,\eta)\,\mathrm{d}\xi\mathrm{d}\eta\right] = O\big(h_x^2h_y^2\big).
\]
Similarly,
\[
\mathbb{E}\left[\bigg|\iint_{\Omega_{\text{low}}}\!\!\!\!G_{(l_1,\,l_2)}(\xi,\eta)\,\mathrm{d}\xi\mathrm{d}\eta\bigg|^2\right] = O\big(h_x^4h_y^4\big).
\]

\end{proof}

\subsection{Proof of Lemma \ref{lem_2dSGridmidwave}}\label{app_lem3.2}

\begin{proof}
As 
\[
\Omega_{\text{mid}} = \big\{(\xi,\eta):\vert\xi\vert\leq h_x^{-p}\text{ and } h_y^{-p}<\vert\eta\vert\leq \pi/h_y\big\} \cup \big\{(\xi,\eta):h_x^{-p}<\vert\xi\vert\leq \pi/h_x\text{ and } \vert\eta\vert\leq h_y^{-p}\big\},
\]
it is enough to show for $\Omega_{\text{mid}}^1 = \{ 0<|\xi|<h_x^{-p},\ h_y^{-p}<|\eta|<\pi/h_y\}$,
\[
 \mathbb{E}\Bigg[\,\bigg|\iint_{\Omega_{\text{mid}}^1} \!\!\!\!\! h_xh_yX_N(\xi,\eta)\chi(h_x,\xi)\chi(h_y,\eta)\,\mathrm{d}\xi\,\mathrm{d}\eta\bigg|^2\,\Bigg] = o(h_y^r),
\]

We have
\begin{align*}
 \mathbb{E}\Bigg[\,\bigg|\iint_{\Omega_{\text{mid}}^1} \!\!\!\!\! h_xh_yX_N(\xi,\eta)\chi(h_x,\xi)&\chi(h_y,\eta)\,\mathrm{d}\xi\,\mathrm{d}\eta\bigg|^2\,\Bigg] \leq 2\,\mathbb{E}\Bigg[\,\bigg|\iint_{\Omega_{\text{mid}}^1} \!\!\!\!\! h_xh_yX(T,\xi,\eta)\chi(h_x,\xi)\chi(h_y,\eta)\,\mathrm{d}\xi\,\mathrm{d}\eta\bigg|^2\,\Bigg]\\
 & + 2\,\mathbb{E}\Bigg[\,\bigg|\iint_{\Omega_{\text{mid}}^1} \!\!\!\!\! h_xh_y\Big(X(T,\xi,\eta) - X_N(\xi,\eta)\Big)\chi(h_x,\xi)\chi(h_y,\eta)\,\mathrm{d}\xi\,\mathrm{d}\eta\bigg|^2\,\Bigg],
\end{align*}
where $X(T)$ is introduced in \eqref{eq_solXn}, and $X_N$ is given by \eqref{eq_Xn2}. Specifically,
\begin{align*}
X(T) &= \prod_{n=0}^{N-1}\exp\Big(-\frac{1}{2}(1-\rho_x)\xi^2k -\frac{1}{2}(1-\rho_y)\eta^2k -\mathrm{i}\xi\sqrt{\rho_xk}Z_{n,x} - \mathrm{i}\eta\sqrt{\rho_yk}\widetilde{Z}_{n,y}\Big),\\
X_N &= \prod_{n=0}^{N-1}\frac{1 -\mathrm{i}c_x\sqrt{\rho_xk}Z_{n,x} -\mathrm{i}c_y\sqrt{\rho_yk}\widetilde{Z}_{n,y} + b_x\rho_xk(Z_{n,x}^2-1) + b_y\rho_yk(\widetilde{Z}_{n,y}^2-1) + d\sqrt{\rho_x\rho_y}kZ_{n,x}\widetilde{Z}_{n,y}}{1-(a_x+a_y)k}.
\end{align*}

As $X(T)$ is given in closed form, a direct calculation gives 
\[
\mathbb{E}\Bigg[\,\bigg|\iint_{\Omega_{\text{mid}}^1} \!\!\!\!\! h_xh_yX(T,\xi,\eta)\chi(h_x,\xi)\chi(h_y,\eta)\,\mathrm{d}\xi\,\mathrm{d}\eta\bigg|^2\,\Bigg] = o(h_y^r),\quad \forall r>0.
\]
Hence, in the following we focus on
\[
\mathbb{E}\Bigg[\,\bigg|\iint_{\Omega_{\text{mid}}^1} \!\!\!\!\! h_xh_y\Big(X(T,\xi,\eta) - X_N(\xi,\eta)\Big)\chi(h_x,\xi)\chi(h_y,\eta)\,\mathrm{d}\xi\,\mathrm{d}\eta\bigg|^2\,\Bigg].
\]
Since $h_xh_y\chi(h_x,\xi)\chi(h_y,\eta)$ has finite 2-norm, it is justifiable to consider
\[
\mathbb{E}\Bigg[\,\bigg|\iint_{\Omega_{\text{mid}}^1} \!\!\!\!\! X(T,\xi,\eta) - X_N(\xi,\eta)\,\mathrm{d}\xi\,\mathrm{d}\eta\bigg|^2\,\Bigg].
\]

First we denote
\[
X_1(T,\xi)\coloneqq \exp\bigg(-\frac{1}{2}(1-\rho_x)\xi^2T -\mathrm{i}\xi\sqrt{\rho_xk}\sum_{n=0}^{N-1}Z_{n,x}\bigg),\ 
X_2(T,\eta)\coloneqq \exp\bigg(-\frac{1}{2}(1-\rho_y)\eta^2T - \mathrm{i}\eta\sqrt{\rho_yk}\sum_{n=0}^{N-1}\widetilde{Z}_{n,y}\bigg),
\]
so that
\[
X(T) = X_1(T,\xi)X_2(T,\eta).
\]

Then we introduce $\widetilde{X}_n$ such that
\[
\widetilde{X}_N = \prod_{n=0}^{N-1}\widetilde{C}_n^x\prod_{n=0}^{N-1}\widetilde{C}_n^y,
\]
where
\begin{align*}
\widetilde{C}_n^x = \frac{1 -\mathrm{i}c_x\sqrt{\rho_xk}Z_{n,x} + b_x\rho_xk(Z_{n,x}^2-1)}{1-a_xk},\quad \widetilde{C}_n^y =\frac{1 -\mathrm{i}c_y\sqrt{\rho_yk}\widetilde{Z}_{n,y} + b_y\rho_yk(\widetilde{Z}_{n,y}^2-1)}{1-a_yk}.
\end{align*}
It has been proved in \cite{reisinger2018analysis} that
\[
X_1(T,\xi) - \prod_{n=0}^{N-1}\widetilde{C}_n^x = X_1(T,\xi)\cdot O(h_x^2),
\]
\[
\mathbb{E}\bigg[\int_{h_y^{-p}}^{\pi/h_y} \bigg|\prod_{n=0}^{N-1}\widetilde{C}_n^y\bigg|^2\,\mathrm{d}\eta \bigg] = o(h_y^r),\quad 
\mathbb{E}\bigg[\int_{h_y^{-p}}^{\pi/h_y} \Big|X_2(T,\eta)\Big|^2\,\mathrm{d}\eta \bigg] = o(h_y^r),\qquad \forall r>0.
\]
Thus there exists a constant $C$ such that
\[
\int_{-h_x^{-p}}^{h_x^{-p}} \widetilde{X}_N - X(T,\xi,\eta)\,\mathrm{d}\xi= C\cdot \big(1+O(h_x^2)\big)\prod_{n=0}^{N-1}\widetilde{C}_n^y - C\cdot X_2(T,\eta),
\]
and as a result,
\begin{align*}
\mathbb{E}\Bigg[\,\bigg|\iint_{\Omega_{\text{mid}}^1} \!\!\!\!\! X(T,\xi,\eta) - \widetilde{X}_N(\xi,\eta)\,\mathrm{d}\xi\mathrm{d}\eta\bigg|^2\,\Bigg] \leq \frac{\pi}{h_y}\int_{h_y^{-p}}^{\pi/h_y} \mathbb{E}\bigg[\,\bigg|\int_{-h_x^{-p}}^{h_x^{-p}}  \!\!\!\!\! X(T,\xi,\eta) - \widetilde{X}_N(\xi,\eta)\,\mathrm{d}\xi\bigg|^2\bigg]\mathrm{d}\eta = o(h_y^r).
\end{align*}

Since
\[
\mathbb{E}\Big[\Big| X_n-\widetilde{X}_n \Big|^2\Big]= O(h_x)\mathbb{E}\big[\big|\widetilde{X}_n\big|^2\big],
\]
it follows that
\begin{align*}
\mathbb{E}\bigg[\,\bigg|\iint_{\Omega_{\text{mid}}^1} \widetilde{X}_N(\xi,\eta) - X_N(\xi,\eta)\,\mathrm{d}\xi\mathrm{d}\eta\bigg|^2\,\bigg]
&\leq \pi h_x^{-p}h_y^{-1} \iint_{\Omega_{\text{mid}}^1} \mathbb{E}\bigg[\Big|\widetilde{X}_N(\xi,\eta) - X_N(\xi,\eta)\Big|^2\bigg]\,\mathrm{d}\xi\mathrm{d}\eta\\
&= \pi h_x^{-p}h_y^{-1} \iint_{\Omega_{\text{mid}}^1} O(h_x)\mathbb{E}\Big[\big|\widetilde{X}_n\big|^2\Big]\,\mathrm{d}\xi\mathrm{d}\eta =o(h_y^r).
\end{align*}

Therefore,
\begin{align*}
&\quad\mathbb{E}\bigg[\,\bigg|\iint_{\Omega_{\text{mid}}^1}\!\!\!\!\! X(T,\xi,\eta) - X_N(\xi,\eta)\,\mathrm{d}\xi\mathrm{d}\eta\bigg|^2\,\bigg] \\
&\leq 2\,\mathbb{E}\bigg[\,\bigg|\iint_{\Omega_{\text{mid}}^1}\!\!\!\!\! X(T,\xi,\eta) - \widetilde{X}_N(\xi,\eta)\,\mathrm{d}\xi\mathrm{d}\eta\bigg|^2\,\Bigg] 
+ 2\,\mathbb{E}\bigg[\,\bigg|\iint_{\Omega_{\text{mid}}^1}\!\!\!\!\! \widetilde{X}_N(T,\xi,\eta) - X_N(\xi,\eta)\,\mathrm{d}\xi\mathrm{d}\eta\bigg|^2\,\bigg] 
=o(h_y^r). 
\end{align*}

\end{proof} 

\subsection{Proof of Lemma \ref{lem_2dSGridhighwave}}\label{app_lem3.3}

\begin{proof}
For the same reason as in the proof of Lemma \ref{lem_2dSGridmidwave}, it is sufficient to prove
\[
\mathbb{E}\Bigg[\,\bigg|\iint_{\Omega_{\text{high}}} X(T,\xi,\eta) - X_N(\xi,\eta)\,\mathrm{d}\xi\mathrm{d}\eta\bigg|^2\,\Bigg] = o(h_x^rh_y^r).
\]

Note that
\[
\Omega_{\text{high}} = \big\{(\xi,\eta): h_x^{-p}<\vert\xi\vert\leq \frac{\pi}{2h_x}\text{ and } h_y^{-p}<\vert\eta\vert\leq \frac{\pi}{2h_y}\big\}.
\]

By \eqref{eq_Xn2}, we have
\begin{align*}
X_N 
= X_0\prod_{n=0}^{N-1}\frac{1 -\mathrm{i}c_x\sqrt{\rho_xk}Z_{n,x} -\mathrm{i}c_y\sqrt{\rho_yk}\widetilde{Z}_{n,y} + b_x\rho_xk(Z_{n,x}^2-1) + b_y\rho_yk(\widetilde{Z}_{n,y}^2-1) + d\sqrt{\rho_x\rho_y}kZ_{n,x}\widetilde{Z}_{n,y}}{1-(a_x+a_y)k}.
\end{align*}
We denote 
\[
u = \frac{\sin^2 \frac{h_x\xi}{2}}{(\frac{h_x\xi}{2})^2}= O(1),\qquad v = \frac{\sin^2 \frac{h_y\eta}{2}}{(\frac{h_y\eta}{2})^2}= O(1).
\]

In this case,
\begin{align*}
\mathbb{E}\big[X_N\big] &= X_0\prod_{n=0}^{N-1} \mathbb{E}\big[ C_n\big] = X_0\bigg(\frac{1+d\sqrt{\rho_x\rho_y}\rho_{xy}k}{1+\frac{1}{2}(\xi^2u+\eta^2v)k}\bigg)^N\\
&< X_0\exp\bigg(-\frac{1}{2}\big(\xi^2u+\eta^2v
+2\xi\eta\frac{\sin\xi h_x\sin\eta h_y}{\xi h_x\eta h_y}\sqrt{\rho_x\rho_y}\rho_{xy}\big)\,T\bigg),\\[4pt]
\mathbb{E}\big[\big|X_N\big|^2\big] &= \big|X_0\big|^2\prod_{n=0}^{N-1} \mathbb{E}\big[ \big|C_n\big|^2\big]\\ 
&\leq \big|X_0\big|^2\bigg(1-\frac{k\xi^2u\big(1-\rho_x\cos^2\frac{\xi h_x}{2}+\frac{1}{4}\xi^2uk(1-2\rho_x^2(1+2|\rho_{xy}|))\big)}{\big(1-(a_x+a_y)k\big)^2}\\
&\qquad\qquad-\frac{k\eta^2v\big(1-\rho_y\cos^2\frac{\eta h_y}{2}+\frac{1}{4}\eta^2vk(1-2\rho_y^2(1+2|\rho_{xy}|))\big)}{\big(1-(a_x+a_y)k\big)^2}\\
&\qquad\qquad -\frac{k^2\xi^2\eta^2uv\big(1-2\rho_x\rho_y(1 +|\rho_{xy}| + 3\rho_{xy}^2)\big)}{2\big(1-(a_x+a_y)k\big)^2}\bigg)^N\\
&\leq \big|X_0\big|^2\bigg(1-\frac{k\xi^2u\big(\beta+\frac{1}{4}\xi^2uk\beta\big)}{\big(1-(a_x+a_y)k\big)^2} 
- \frac{k\eta^2v\big(\beta+\frac{1}{4}\eta^2vk\beta\big)}{\big(1-(a_x+a_y)k\big)^2} 
- \frac{k^2\xi^2\eta^2uv\beta}{2\big(1-(a_x+a_y)k\big)^2}\bigg)^N\\
&< \big|X_0\big|^2\exp\bigg(-\xi^2u\beta T \frac{1+\frac{1}{4}\beta\xi^2uk}{\big(1+\frac{1}{2}(\xi^2u+\eta^2v)k\big)^2}
- \eta^2v\beta T\frac{1 +\frac{1}{4}\beta\eta^2vk}{\big(1+\frac{1}{2}(\xi^2u+\eta^2v)k\big)^2}\bigg),
\end{align*}
where
\[
\beta = \min\big\{1-\rho_x, 1-\rho_y, 1-2\rho_x^2(1+2|\rho_{xy}|), 1-2\rho_y^2(1+2|\rho_{xy}|), 1-2\rho_x\rho_y(1+2|\rho_{xy}|+3\rho_{xy}^2)\big\} \in (0,1).
\]
As we have $k<\lambda \min\{h_x^2,h_y^2\}$,
\[
0< \xi^2k < \lambda \pi^2/4,\quad 0< \eta^2k < \lambda \pi^2/4.
\]
Also,
\[
\min\Big\{\frac{u(1+\frac{1}{4}\beta\xi^2uk)}{\big(1+\frac{1}{2}(\xi^2u+\eta^2v)k\big)^2}, \frac{v(1+\frac{1}{4}\beta\eta^2vk)}{\big(1+\frac{1}{2}(\xi^2u+\eta^2v)k\big)^2}\Big\} > \frac{8}{\pi^2\big(1+\frac{1}{4}\lambda \pi^2\big)^2}.
\]

We define a constant
\[
\kappa \coloneqq  \frac{8\beta T}{\pi^2\big(1+\frac{1}{4}\lambda \pi^2\big)^2},
\]
then
\[
\mathbb{E}\big[\big|X_N\big|^2\big] < |X_0|^2\exp\big(-\kappa \xi^2 - \kappa \eta^2 \big).
\]

Therefore we have
\begin{align*}
&\quad\mathbb{E}\Bigg[\,\bigg|\iint_{\Omega_{\text{high}}} X_N(\xi,\eta)\,\mathrm{d}\xi\mathrm{d}\eta\bigg|^2\,\Bigg]< \frac{\pi^2}{h_xh_y} \iint_{\Omega_{\text{high}}} \mathbb{E}\Big[\big|X_N(\xi,\eta)\big|^2\Big]\,\mathrm{d}\xi\mathrm{d}\eta < \frac{\pi^2|X_0|^2}{h_xh_y}  \iint_{\Omega_{\text{high}}}\mathrm{e}^{-\kappa\xi^2 - \kappa\eta^2}\,\mathrm{d}\xi\mathrm{d}\eta\\
& = \frac{\pi^2|X_0|^2}{h_xh_y} \int_{h_y^{-p}}^{\pi/h_y}\int_{h_x^{-p}}^{\pi/h_x}\mathrm{e}^{-\kappa\xi^2 - \kappa\eta^2}\,\mathrm{d}\xi\mathrm{d}\eta 
\leq C\cdot (h_xh_y)^{-1+p}\exp\big(-\kappa h_x^{-p} -\kappa h_y^{-p}\big) = o(h_x^rh_y^r),\quad \forall r>0.
\end{align*}

\end{proof}


\bibliography{references}
\bibliographystyle{plain}

\end{document}